\documentclass[10pt,a4paper]{article}
\usepackage[a4paper,left=2cm,right=2cm,top=2cm,bottom=2cm]{geometry}
\usepackage[thinlines]{easytable}
\usepackage{booktabs}
\usepackage{array}
\newcolumntype{x}[1]{>{\centering\arraybackslash\hspace{0pt}}p{#1}}
\newcolumntype{v}[1]{>{\arraybackslash\hspace{0pt}}p{#1}}
\newcolumntype{g}[1]{>{\columncolor{gray!10}\centering\arraybackslash\hspace{0pt}}p{#1}}
\usepackage[svgnames,table]{xcolor}

\usepackage[backend=bibtex,bibencoding=ascii,giveninits=true,url=false,isbn=false,doi=false,style=alphabetic,maxnames=5]{biblatex}
\usepackage[utf8]{inputenc}
\usepackage{graphicx}
\usepackage{epstopdf}

\usepackage{amsmath,amsthm}
\usepackage{mathtools}
\usepackage{amsfonts}
\usepackage{amssymb}
\usepackage{nicefrac}
\usepackage{dsfont}
\usepackage[inline]{enumitem}
\usepackage{afterpage}

\usepackage{tikz}
\usepackage{pgfplots}
\usepgfplotslibrary{patchplots,colormaps}
\pgfplotsset{compat=1.9}

\definecolor{linkcolor}{rgb}{0.5,0.0,0.0}
\definecolor{citecolor}{rgb}{0.0,0.5,0.0}
\definecolor{urlcolor} {rgb}{0.0,0.0,0.5}

\theoremstyle{plain}

\newtheorem{lemma}{Lemma}
\newtheorem{theorem}{Theorem}
\newtheorem{proposition}{Proposition}
\newtheorem{corollary}{Corollary}

\theoremstyle{definition}
\newtheorem{definition}{Definition}

\newtheorem{example}{Example}

\theoremstyle{definition}
\newtheorem{remark}{Remark}

\newcommand{\cc}[1]{\ensuremath{\overline{#1}}}

\newcommand{\pe}[1]{\ensuremath{\check{#1}}}

\newcommand{\g}[1]{\ensuremath{g_{#1}}}

\newcommand{\projalg}{\ensuremath{\mathfrak{p}}}
\newcommand{\projcl}{\ensuremath{\mathfrak{P}}}
\newcommand{\solSp}{\ensuremath{\mathfrak{A}}}

\newcommand{\metrSp}{\ensuremath{\mathfrak{M}}}
\newcommand{\essSp}{\ensuremath{\mathfrak{E}}}

\newcommand{\lie}{\ensuremath{\mathcal{L}}}

\newcommand{\h}[2]{\begin{minipage}{3.5cm}\centering \textbf{#1}\\ \textit{#2} \end{minipage}}

\newcommand{\Rnz}{\ensuremath{\mathbb{R}\setminus\{0\}}}
\newcommand{\pmo}{\ensuremath{ \{\pm1\} }}
\newcommand{\angl}{\ensuremath{[0,2\pi)}}
\newcommand{\halfangl}{\ensuremath{[0,\pi)}}

\newcommand{\R}{\ensuremath{\mathbb{R}}}

\DeclareMathOperator{\sgn}{sgn}

\DeclareMathOperator{\erfi}{erfi}
\DeclareMathOperator{\erf}{erf}
\DeclareMathOperator{\hypergeom}{hypgeom}

\DeclareMathOperator{\rk}{rk}
\DeclareMathOperator{\Id}{Id}

\makeatletter
\newcommand{\allowhy}{\nobreak\hskip\z@skip}
\makeatother

\makeatletter
\newcommand{\pushright}[1]{\ifmeasuring@#1\else\omit\hfill$\displaystyle#1$\fi\ignorespaces}
\newcommand{\pushleft}[1]{\ifmeasuring@#1\else\omit$\displaystyle#1$\hfill\fi\ignorespaces}
\makeatother

\begin{document}

\title{(Super-)integrable systems associated to 2-dimensional\\ projective connections with one projective symmetry}
\author{%
	\textsc{Gianni Manno} (\textsf{giovanni.manno@polito.it})\\
    \footnotesize{
    Dipartimento di Scienze Matematiche (DISMA),
    Politecnico di Torino,
    Corso Duca degli Abruzzi, 24,
    10129 Torino, Italy
    },\\[0.5cm]
    \textsc{Andreas Vollmer} (\textsf{andreasdvollmer@gmail.com})\\
    \footnotesize{School of Mathematics and Statistics,
    University of New South Wales,
    Sydney NSW 2052, Australia
    }
    }
\maketitle

\begin{abstract}
\noindent Projective connections arise from equivalence classes of affine connections under the reparametrization of geodesics. They may also be viewed as quotient systems of the classical geodesic equation.
After studying the link between integrals of the (classical) geodesic flow and its associated projective connection, we turn our attention to 2-dimensional metrics that admit one projective vector field, i.e.\ whose local flow sends unparametrized geodesics into unparametrized geodesics.
We review and discuss the classification of these metrics, introducing special coordinates on the linear space of solutions to a certain system of partial differential equations, from which such metrics are obtained.
Particularly, we discuss those that give rise to free second-order superintegrable Hamiltonian systems, i.e.\ which admit 2 additional, functionally independent quadratic integrals.
We prove that these systems are parametrized by the 2-sphere, except for 6 exceptional points where the projective symmetry becomes homothetic.
\end{abstract}


\section{Introduction}

The term \emph{metric} is used for metrics of arbitrary signature. Specifically, in the 2-dimensional case, the word metric stands for both Riemannian and Lorentzian metrics, unless otherwise specified. Einstein's summation convention will be used.
Our focus is a local study, i.e.\ the manifold may be identified with an open subset of~$\mathbb{R}^2$. This is in line with the existing literature on the topic, e.g. compare~\cite{bryant_2008,matveev_2012,manno_2018}, which are also of a local nature.
Let us now introduce the basic terminology and the main objects used in this paper.

\begin{definition}
We say that two symmetric affine connections on the same manifold $M$ are \emph{projectively equivalent} if they share the same geodesics (as unparametrized curves). The set of all connections projectively equivalent to a given connection $\nabla$  is called the \emph{projective class} of such a connection, denoted by $[\nabla]$.

If a projective class $[\nabla]$ contains a Levi-Civita connection $\nabla^g$ of a metric~$g$, we say that $[\nabla]$ is \emph{metrizable}. Two metrics are said to be projectively equivalent if their Levi-Civita connections are so. We denote by~$\projcl(g)$ the projective class of~$g$, i.e., the set of metrics projectively equivalent to $g$.
\end{definition}

\noindent A classical result \cite{weyl_1921,thomas_1925,veblen_1922} states that two connections $\nabla$ and $\widetilde{\nabla}$ are projectively equivalent if there exits a 1-form $\alpha$ such that
\begin{equation}\label{eqn:prj.equ.conn}
\widetilde{\nabla} - \nabla = \alpha\otimes\Id + \Id\otimes\alpha\,,
\end{equation}
or, locally,
\begin{equation}\label{eqn:Gamma.bar.Gamma}
\widetilde{\Gamma}^a_{bc}-\Gamma^a_{bc}=\delta^a_b\alpha_c + \delta^a_c\alpha_b\,.
\end{equation}
where $\widetilde{\Gamma}^a_{bc}$ and $\Gamma^a_{bc}$ are, respectively, the Christoffel symbols of $\widetilde{\nabla}$ and $\nabla$.
For a more detailed explanation and proof of this statement see, for instance, \cite{matveev_2012rel}.

\smallskip
\begin{definition}\label{def:projective.symmetry}
A \emph{projective transformation} is a (local) diffeomorphism of $M$ that sends geodesics into geodesics (where geodesics are to be understood as unparametrized curves).
A vector field on $M$ is \emph{projective} if its (local) flow acts by projective transformations.
\end{definition}

\noindent The term \emph{projective symmetry}, as defined in Definition~\ref{def:projective.symmetry}, is explained for a given connection. We say that a projective symmetry~$X$ is homothetic (for a metric~$g$), if $\lie_Xg=\lambda g$ for some $\lambda\in\mathbb{R}$. Particularly, if $X$ is Killing, it is also homothetic. If $X$ is not homothetic, it is called \emph{essential}.

While the characterization of a projective symmetry depends on the projective class only, the property of being homothetic or essential relies on the existence of a metric.
%
Oftentimes, one is interested in projective classes that admit metrics with an essential projective symmetry, see e.g.~\cite{matveev_2012}. This is also true in the present paper, and we therefore introduce the concept of essentiality of projective symmetries on the level of projective connections.
\begin{definition}
	An infinitesimal projective symmetry~$X$ of a metrizable projective connection $[\nabla]$ is called \emph{essential} if and only if the symmetry~$X$ is non-homothetic for at least one metric~$g$ that realizes $[\nabla]$ via its Levi-Civita connection.
\end{definition}

\noindent In view of formula \eqref{eqn:prj.equ.conn}, a vector field $X$ is a projective vector field if there exists a 1-form $\mu$ such that
\begin{equation}\label{eq:proj.vect.field}
\lie_X\nabla = \mu\otimes\Id + \Id\otimes\mu
\end{equation}
where the Lie derivative $\lie_X\nabla$ of $\nabla$ along $X$ is a $(1,2)$-tensor defined as follows:
\begin{equation*}
(\lie_X\nabla)(Y,Z):=\lie_X(\nabla_YZ)-\nabla_Y(\lie_XZ)-\nabla_{\lie_XY}Z=[X,\nabla_YZ]-\nabla_Y[X,Z]-\nabla_{[X,Y]}Z\,.
\end{equation*}
Locally, \eqref{eq:proj.vect.field} reads
\begin{equation*}
\lie_X\Gamma^i_{jk}=\mu_j\delta^i_k+\mu_k\delta^i_j\,.
\end{equation*}

\noindent From an ODE perspective, the projective class of a given symmetric affine connection $\nabla$ can be understood as follows. Let $(y^1,y^2,\dots,y^{N})=(x,y^2,\dots,y^{N})$ be a system of coordinates on $M$. Then $\nabla$ gives rise to a system of second order ordinary differential equations
\begin{equation}\label{eq:proj.conn.gen.multidim}
y^k_{xx}=-\Gamma_{11}^k - (2\Gamma_{1i}^k-\delta^k_i\Gamma_{11}^1)  y^i_x - (\Gamma_{ij}^k  - 2\delta^k_i \Gamma_{1j}^1 )
y^i_x y^j_x + \Gamma_{ij}^1 \, y^i_x y^j_x y^k_x\,,\quad k=2,\dots,N\,,\,\,\, i,j\geq 2\,,
\end{equation}
obtained by eliminating the external parameter from the classical geodesic equations (see for instance \cite{manno_advances}):
\begin{equation}\label{eq:param.geod.equ}
\ddot{y}^k+\Gamma^k_{ij}\dot{y}^i\dot{y}^j=0\,, \quad k=1,\dots,N\,.
\end{equation}

\begin{definition}
System~\eqref{eq:proj.conn.gen.multidim} is called the~\emph{projective connection} representing the projective class~$[\nabla]$.
\end{definition}
\noindent In Section \ref{sec:ode.perspecitve} we describe in detail how to construct this representative system.\medskip

\noindent For any solution $(y^2(x),\dots,y^N(x))$ to \eqref{eq:proj.conn.gen.multidim}, the curve $(x,y^2(x),\dots,y^N(x))$ is, up to reparametrization, a solution to \eqref{eq:param.geod.equ}, i.e., a geodesic of $\nabla$. Thus, we can interpret system \eqref{eq:proj.conn.gen.multidim} as  the projective connection associated to $\nabla$. On the other hand, as we said, two connections $\nabla$ and $\widetilde{\nabla}$ are projectively equivalent if and only if they are related by~\eqref{eqn:Gamma.bar.Gamma} and connections linked by~\eqref{eqn:Gamma.bar.Gamma} give the same system of equations~\eqref{eq:proj.conn.gen.multidim}.

\medskip
\noindent Taking into account the above considerations, finite point symmetries, i.e.\ local diffeomorphisms
\[
 (y^1,\dots,y^N)\to\big(y^1(v^1,\dots,v^N),\dots,y^N(v^1,\dots,v^N)\big)
\]
that preserve~\eqref{eq:proj.conn.gen.multidim}, are precisely the projective transformations of the connection~$\nabla$, as they send solutions (i.e., unparametrized geodesics) into solutions. Infinitesimal point symmetries of \eqref{eq:proj.conn.gen.multidim} are projective vector fields of~$\nabla$ and generate a $1$-parametric family of projective transformations. We can see a ($N$-dimensional) projective connection as
the following system of ODE:
\begin{equation}\label{eq:proj.conn.gen.multidim.2}
y^k_{xx}=f_{11}^k + f_{1i}^k  y^i_x + f_{ij}^k
y^i_x y^j_x + f^1_{ij} \, y^i_x y^j_x y^k_x\,,\quad k=2,\dots,N\,,\,\,\, i,j\geq 2\,,\quad f^k_{ij}=f^k_{ji}\,.
\end{equation}
Note that, for $N=2$, system \eqref{eq:proj.conn.gen.multidim} reduces to a single ODE, namely the classical $2$-dimensional projective connection associated to a $2$-dimensional metric
\begin{equation*}
y_{xx} = -\Gamma^2_{11} +(\Gamma^1_{11}-2\Gamma^2_{12})\,y_x -(\Gamma^2_{22}-2\Gamma^1_{12})\,y_x^2 +\Gamma^1_{22}\,y_x^3
\end{equation*}
where $(x,y,y_x,y_{xx}):=(y^1,y^2,y^2_x,y^2_{xx})$. Furthermore, \eqref{eq:proj.conn.gen.multidim.2} reduces to a single ODE
\begin{equation}\label{eqn:predefined.proj.connnection}
  y_{xx} = f_0 +f_1\,y_x +f_2\,y_x^2 +f_3\,y_x^3\,,\quad f_i=f_i(x,y)\,,
\end{equation}
extensively studied, for instance, in \cite{bryant_2008,bryant_2009,matveev_2012,sharpe_1997,aminova_2006}.

\section{Metrizable projective connections and Benenti tensors}\label{sec:connections.and.benenti.tensors}

A classical question is whether a projective connection is metrizable.
%
In local coordinates, the projective connection \eqref{eq:proj.conn.gen.multidim.2} is metrizable if there exists an $N$-dimensional metric $g$ such that \eqref{eq:proj.conn.gen.multidim}, where $\Gamma^{i}_{jk}$ are the Christoffel symbols of the Levi-Civita connection of $g$, is equal to \eqref{eq:proj.conn.gen.multidim.2}. This is equivalent to the existence of a solution to the following system of $\frac12 N(N-1)(N+2)$ partial differential equations (PDE):
\begin{multline}\label{eq:metr.cond.non.linear}
-\Gamma_{11}^k=f_{11}^k \,,\,\,\, - (2\Gamma_{1m}^k-\delta^k_m\Gamma_{11}^1)=  f_{1m}^k  \,,\,\,\,  -\big((2-\delta_{jm})\Gamma_{jm}^k  - 2\delta^k_m \Gamma_{1j}^1 - 2\delta^k_j \Gamma_{1m}^1\big)= f_{jm}^k\,,\,\,\,
  \Gamma_{ij}^1 =f^1_{ij} \,,
  \\
  k,i,j,m=2,\dots,N
\end{multline}
System \eqref{eq:metr.cond.non.linear} is highly non-linear in the unknown functions $g_{ij}$, but it turns out that if we perform the substitution
\begin{equation}\label{eqn:sigma.g}
\sigma^{ij}:=\det(g)^{\frac{1}{N+1}}g^{ij}\in S^2(M)\otimes (\Lambda^N(M))^{\frac{2}{N+1}}\,,
\end{equation}
we obtain a \emph{linear} system in the unknown variables $\sigma^{ij}$.
Of course, \eqref{eqn:sigma.g} does not make sense if $g$ is negative-definite. Therefore, without further mentioning, in \eqref{eqn:sigma.g} and in the remainder of the paper, the fractional exponent implies that we use the absolute value of the base expression unless stated otherwise. We have the following theorem.
\begin{theorem}[\cite{eastwood_2008}]\label{thm:east.mat}
A metric $g$ on an $N$-dimensional manifold lies in the projective class of a given connection~$\nabla$ if and only if $\sigma^{ij}$ defined by \eqref{eqn:sigma.g} is a solution to
\begin{equation}\label{eqn:linear.system}
\nabla_a\sigma^{bc}-\frac{1}{N+1}(\delta^c_a\nabla_i\sigma^{ib} + \delta^b_a\nabla_i\sigma^{ic} )=0\,.
\end{equation}
\end{theorem}

\noindent Note that Theorem \ref{thm:east.mat} implicitly contains the assumption $\det(\sigma)\ne0$, as $\sigma^{ij}$ otherwise does not correspond to a metric. Furthermore, note that, since $\sigma$ is a weighted tensor field,
\begin{equation*}
\nabla_a\sigma^{bc} = \sigma^{bc}_{,a} + \Gamma^b_{ad}\sigma^{dc} + \Gamma^c_{ad}\sigma^{db} - \frac{2}{N+1}\Gamma^d_{da}\sigma^{bc}\,.
\end{equation*}
Of course, the set of solutions to the linear system of PDEs \eqref{eqn:linear.system} form a linear space.
As a special case of Theorem~\ref{thm:east.mat} we have, in dimension~2, the following example.
\begin{example}[\cite{bryant_2008,liouville_1889}]
	The projective connection associated to the Levi-Civita connection of a metric~$g$ is~\eqref{eqn:predefined.proj.connnection}
	if and only if the entries $\sigma^{ij}$ of the matrix $\sigma=|\det(g)|^{\nicefrac13}g^{-1}$ satisfy the linear system of PDEs
	\begin{equation}\label{eqn:linear.system.2D}
	\left.
	\begin{array}{rl}
	\sigma^{22}_x - \frac23\,f_1\,\sigma^{22}-2f_0\,\sigma^{12} &= 0
	\\
	\sigma^{22}_y - 2\sigma^{12}_x -\frac43\,f_2\,\sigma^{22} -\frac23\,f_1\,\sigma^{12}+2f_0\,\sigma^{11} &= 0
	\\
	-2\sigma^{12}_y + \sigma^{11}_x -2f_3\,\sigma^{22} +\frac23\,f_2\,\sigma^{12}+\frac43\,f_1\,\sigma^{11} &= 0
	\\
	\sigma^{11}_y + 2f_3\,\sigma^{12}+\frac23\,f_2\,\sigma^{11} &= 0
	\end{array}
	\right\}
	\end{equation}
	where the subscripts $x,y$ denote derivatives.
\end{example}

\begin{remark}
	In dimension~2, it is also possible to linearize the system~\eqref{eqn:linear.system.2D} by choice of the new unknowns~$a_{ij}$, i.e.\ components of the matrix $a=\frac{g}{|\det(g)|^{\nicefrac23}}$. In this case, the object $a$ should be understood as a section of \smash{$S^2(M)\otimes(\Lambda^2(M))^{-\frac43}$}.
\end{remark}

\begin{definition}\label{def:metrization.space}
Let $g$ be a metric and $\nabla$ its Levi-Civita connection. The (linear) space of solutions to system  \eqref{eqn:linear.system} is denoted by $\solSp(g)$ (or $\solSp$ in short when there is no risk of confusion) and it is called the \emph{Liouville metrization space}, or briefly \emph{Liouville space}. Its dimension is called the \emph{degree of mobility of $g$}.
The subset $\metrSp\subset\solSp$ of solutions that actually correspond to a metric, projectively equivalent to the underlying metric~$g$, is to be called the \emph{metrizability space} and denoted by~$\metrSp$.
\end{definition}

\begin{remark}
In the case of a 2-dimensional manifold, the Liouville metrization and the metrizability space are almost identical, differing only by the origin, i.e.\ $\metrSp = \solSp\setminus\{0\}$.
This follows from Lemma~2 of~\cite{matveev_2012}, where the following statement is proven:
Assume $\sigma\in\solSp$ is not identically zero. Then, the set of the points where $\sigma$ is degenerate is nowhere dense.
\end{remark}

\noindent Since the degree of mobility is the same for any choice of metric $g$ within a given projective class, it is also reasonable to call the degree of mobility of $g$ the degree of mobility of its projective class (or of its projective connection).

\begin{proposition}[\cite{beltrami_1865,bolsinov_2009,dini_1869}]
	\label{prop:dini}
	Let $\g1$ and $\g2$ be two metrics that are projectively equivalent, but not proportional.
	Then, around almost every point, there exist local coordinates, say $(x,y)$, such that the pair $(\g1,\g2)$ assumes one of the following \emph{Dini normal forms}.
	\smallskip
	
	\begin{table}[!ht]
	\begin{center}
	\textbf{\upshape The Dini normal forms for pairs of projectively equivalent metrics}\smallskip
	
	\begin{tabular}{cg{4cm}x{4.2cm}g{4cm}}
	 \toprule
	 & \h{A}{Liouville case} & \h{B}{complex Liouville case} & \h{C}{Jordan block case} \\
	 \midrule
	 \g1 & $(X-Y)(dx^2\pm dy^2)$ & $(\cc{h(z)}-h(z))\,(d\cc{z}^2-dz^2)$ & $(1+xY^\prime)dxdy$  \\
	 \midrule
	 \g2
	 & $(\frac{1}{X}-\frac{1}{Y})(\frac{dx^2}{X}\pm\frac{dy^2}{Y})$
	 & \begin{minipage}{4.2cm}\centering $\left(\frac{1}{\cc{h(z)}}-\frac{1}{h(z)}\right)\left(\frac{d\cc{z}^2}{\cc{h(z)}}-\frac{dz^2}{h(z)}\right)$ \end{minipage}
	 & \begin{minipage}{3.5cm}\centering $\frac{1+xY'}{Y^4}\ (-2Y\,dxdy$\\ $+(1+xY')\,dy^2)$\end{minipage} \\
	 \bottomrule
	\end{tabular}
	\caption{The Dini normal forms as developed in~\cite{bolsinov_2009}. The functions $X=X(x)$ and $Y=Y(y)$ depend on one variable only, and in the complex Liouville case we use variables $z=x+iy$, $\cc{z}=x-iy$.}\label{tab:dini}
	\end{center}
	\end{table}
\end{proposition}

\noindent Note that, if the degree of mobility is precisely one, instead of Proposition~\ref{prop:dini}, the following normal form holds \cite{lie_1882,bryant_2008,matveev_2012}: There exist local coordinates such that $X=\partial_x$ and every metric in the projective class is proportional to a metric
\[
 g = e^{\lambda x}\,\begin{pmatrix} q_1(y) & q_2(y) \\ q_2(y) & q_3(y) \end{pmatrix}
\]
where $q_i$ are functions in one variable.
On the other hand, if the degree of mobility is at least~3, note that Proposition~\ref{prop:dini} still holds, but depending on the choice of the pair $(\g1,\g2)$ the Dini type might differ. Indeed, this phenomenon is observed in~\cite{manno_2018} in the case of one, essential projective vector field and a metric with degree of mobility~3.

\begin{definition}\label{def:dini.basis}
 The coordinates given in Proposition~\ref{prop:dini} are called \emph{Dini coordinates}.
 A \emph{Dini basis} is a basis of the metrization space, see Definition~\ref{def:metrization.space}, whose corresponding metrics, via Formula~\eqref{eqn:sigma.g}, are of the form as in Proposition~\ref{prop:dini}.
 Analogously, if the degree of mobility is 3, a Dini basis is understood to be the basis corresponding to~\eqref{eqn:supint.generators}.
\end{definition}

\noindent Let $X$ be a projective vector field. Furthermore, let $\sigma$ be a solution of~\eqref{eqn:linear.system}. Then, as is well known from e.g.~\cite{bryant_2008,eastwood_2008}, the Lie derivative $\lie_X\sigma$ is also a solution of~\eqref{eqn:linear.system}. Indeed, the Lie derivative w.r.t.~$X$ acts on the Liouville space~$\solSp$ as a linear action
\begin{equation}\label{eqn:LX.action}
 \lie_X : \solSp\to\solSp\,,\qquad \sigma\to\lie_X\sigma\,.
\end{equation}
Let $\phi_t$ be the (local) flow of $X$, then also the pullback
\[
  \phi_t^* : \solSp\to\solSp\,,\qquad \sigma\to\phi_t^*\sigma
\]
acts on the Liouville space~$\solSp$.

If~$X$ is homothetic for $\sigma$, i.e.\ $\lie_X\sigma=\lambda\sigma$ for some $\lambda\in\mathbb{R}$, then $X$ is also homothetic for $\lie_X\sigma$, i.e.\ the action of $\lie_X$ naturally preserves eigenspaces. This immediately leads to the following observation.
\begin{lemma}\label{la:L.X.nonproportional}
 If $X$ is projective, but not homothetic for $\sigma\in\metrSp$, then $\lie_X\sigma\in\solSp$ is a solution that is non-proportional to~$\sigma$.
\end{lemma}
\noindent This observation allows us, for metrics with an essential projective vector field and with mobility degree~$2$, to easily construct the entire metrization space from the projective vector field $X$ and the metric~$g$ corresponding to $\sigma\in\metrSp$.
This is going to be discussed in more detail below in Proposition~\ref{prop:IX.independence}.

\begin{definition}
For a projective connection with precisely one essential projective vector field (up to rescaling), we define the \emph{essential metrizability space} $\essSp\subset\metrSp$ as the subset of solutions of~\eqref{eqn:linear.system} that correspond to metrics, via Formula~\eqref{eqn:sigma.g}, such that the projective symmetry is essential.
\end{definition}

\noindent Note that, by definition, we have the following tower of inclusions: $\essSp\subseteq\metrSp\subset\solSp$.\smallskip

\noindent Given a pair of solutions to~\eqref{eqn:linear.system}, we can define a $(1,1)$-tensor field (i.e., an endomorphism). Before we give a formal definition, let us briefly digress on the significance of this object, which is known as the Benenti tensor (or sometimes, although ambiguous, special conformal Killing tensor). It was formally introduced, as a $(2,0)$-tensor on Riemannian manifolds, in \cite{benenti_1992}. In later works by S.~Benenti, they are referred to as $L$-tensors, e.g.~\cite{benenti_2005,benenti_2008}.
In~\cite{matveev_2012rel,bolsinov_2009sg,manno_2018ben}, and many other references, Benenti tensors appear as powerful tools, not least for projective differential geometry.
For instance, eigenvalues of the $(1,1)$-Benenti tensor are closely linked to orthogonal separation coordinates (cf., e.g., Killing-St\"ackel spaces).
There exists a number of applications in physics, see e.g.~\cite{matveev_2012rel,crampin_2000} or~\cite{ibort_2000,bolsinov_2003}. The latter references contain the following definition.
\begin{definition}[\cite{bolsinov_2003,ibort_2000}]\label{def:benenti.g}
	Let $g$ be a Riemannian metric on a $N$-dimensional manifold $M$. A non-degenerate $(1,1)$-tensor field $L$ on $M$ is called a \emph{Benenti tensor field with respect to $g$} if $L$ is self-adjoint and satisfies
	\begin{enumerate}
		\item The Nijenhuis tensor of $L$ is identically zero;
		\item For the functions $H:=\frac12 g^{ij}p_ip_j$, $F:=g^{ik}L^j_kp_ip_j$, we have
		\[
		\{H,F\}=2H\left( \frac{\partial\mathrm{trace}(L)}{\partial x^i}g^{ij}p_j \right)
		\]
		where $x^i,p_j$ are the standard coordinates on the cotangent bundle $T^*\!M$ and $\{\,,\,\}$ is the standard Poisson bracket between functions on $T^*\!M$.
	\end{enumerate}
\end{definition}

\noindent In the context of metric projective differential geometry, i.e.\ for a pair $g,g'$ of projectively equivalent metrics, the Benenti tensor can be defined as follows.
\begin{definition}[\cite{bolsinov_2003} and {\cite[Sect.~5]{benenti_2005}}]
 Let $g,\bar{g}$ be projectively equivalent metrics. Their \emph{associated Benenti tensor} is
 \begin{equation}\label{eqn:benenti.tensor}
   L(g,\bar{g}) := \left(\frac{\det(\bar{g})}{\det(g)}\right)^{\frac{1}{N+1}}\,\bar{g}^{-1}g\,.
 \end{equation}
\end{definition}

\noindent Indeed, for Riemannian metrics, the Benenti tensor~\eqref{eqn:benenti.tensor} is a Benenti tensor for $g$ in the sense of Definition~\ref{def:benenti.g}.
Definition~\ref{def:benenti.g} reveals, as well, the close relation of Benenti tensors, projective differential geometry and integrals of motion, i.e.\ conserved quantities in classical mechanics. This relation will be examined more closely in Section~\ref{sec:integration.of.connection}, see also~\cite{topalov_2003,bryant_2008,manno_2018ben} and references therein.

An integral of motion is a function on the cotangent space $I:T^*M\to\mathbb{R}$ (or, analogously, on the tangent space) such that the Poisson bracket with the Hamiltonian $H=\frac12g^{ij}p_ip_j$ vanishes, $\{H,I\}=0$.\footnote{Note that we consider a free Hamiltonian system, i.e.\ the Hamiltonian is determined by the metric~$g$ exclusively. For the case when the Hamiltonian includes also a potential term, i.e.\ $H=\frac12g^{ij}p_ip_j+V$ with a function $V:T^*M\to\mathbb{R}$, see~\cite{valent_explicit_2015,vollmer_2018}, for instance.}
A classical question is how many integrals of motion a given Hamiltonian system admits, and which Hamiltonian systems admit a large number of such integrals. For the purposes of the current paper, we introduce the notion of superintegrability.
\begin{definition}\label{def:superintegrability}
 A Hamiltonian system is said to be \emph{superintegrable} if it admits $2n-1$ functionally independent integrals of motion including the Hamiltonian itself. Equivalently, this can be expressed by the existence of two distinct Benenti tensors, $L_1=L(\sigma,\bar{\sigma})$ and $L_2=L(\sigma,\hat{\sigma})$, for the metric~$g$ corresponding to~$\sigma$.
\end{definition}

\noindent Similarly, we may define the concept of integrable metrics as follows:
A pair $(g,L)$ is called a \emph{$2$-dimensional integrable system} if $g$ is a $2$-dimensional metric and~$L$ a Benenti tensor for it.

A superintegrable system admits a `maximal' number of integrals assuming their functional independence.
It is common to require, in addition, that the integrals are polynomials in the fibre coordinates\footnote{Depending on the context, we will express integrals as polynomials either in momenta or velocities, whichever is more convenient.}.
Let us assume the integral $I$ is a homogeneous polynomial in the~$p_i$ of degree~$d$. It is well known that such integrals correspond to Killing tensors, i.e.\ symmetric $d$-tensor fields $K_{i_1\dots i_d}$ that satisfy \cite{topalov_2003}
\begin{equation}\label{eqn:killing.tensor}
K_{(i_1\dots i_d,i_{d+1})} = 0\,.
\end{equation}
If all integrals are polynomials of at most degree $k$, the superintegrable system is said to be \emph{of order~$k$}, see~\cite{kalnins_2018} for a comprehensive introduction. Second order superintegrable systems on constant curvature manifolds (and conformally flat spaces) are essentially classified (in terms of normal forms) in dimension~2~\cite{kalnins_2001} and~3~\cite{kalnins_III,capel_2014}. Partially, a classification in terms of an algebraic variety has been reached, see~\cite{kress_2018,capel_2014}. Higher order superintegrability is less understood today. Some recent papers on superintegrable systems involving integrals of order~3 are, for instance, \cite{matveev_2011supint,ranada_drach_1997,valent_explicit_2015,marquette_superintegrable_2008,popper_third_2012,tremblay_third_2010}.

\section{Overview of the paper and summary of main outcomes}\label{sec:main}

Sections~\ref{sec:classification} to~\ref{sec:proof.superintegrable} are the main body of this present work.
In Section~\ref{sec:classification}, we introduce some terminology and discuss briefly the strategy that is employed in earlier papers, namely~\cite{matveev_2012,manno_2018} to obtain the normal forms of metrics with one, essential projective symmetry. The section is concluded with an improvement of two of the existing normal forms. Concretely, we remove the remaining ambiguity by introducing normal forms in terms of hypergeometric special functions.
Next, in Section~\ref{sec:projective.vector.fields.integrals}, we review projective connections as quotient systems from an ODE perspective and discuss how integrals of these quotient systems can be obtained from integrals of the geodesic equation of a particular representative of the projective class. We also discuss the relation between integrals and projective symmetries of metrizable projective connections.
Sections~\ref{sec:proof.distinguished.coordinates} and~\ref{sec:proof.superintegrable} provide the proof of the two main theorems, outlined below in the present section.
Finally, in Appendices~\ref{app:normal.forms.dom.2} and~\ref{app:normal.forms.dom.3}, we review the existing classification in terms of normal forms of metrics that admit one, essential projective symmetry.
This classification is going to be referred to in many parts of the present paper.

Before we start with the detailed outline of the main results, let us digress briefly on the terminology that is going to be used in the remainder of the paper. Specifically, concerning terminology, we are faced with the following issue:
We are dealing with three major sets of characteristic data, specifying either a projective class, a metric within a particular projective class (respectively, a solution of~\eqref{eqn:linear.system.2D}), or a point on a differentiable manifold.
In order to avoid ambiguities let us therefore agree, once and for all, on the following terminology.
\begin{enumerate}[font=\bfseries,noitemsep,labelsep=2pt,align=left]
	\item[Variables] specify a point on a differentiable manifold and will typically be called $x,y$, unless stated otherwise.
	\item[Coordinates] reference points within a given metrization space. We denote them, typically, by $(u_1,u_2)$ if they refer to the natural (Dini) basis introduced in Definition~\ref{def:dini.basis}, see also \cite{matveev_2012,manno_2018}.
	\item[Distinguished coordinates] on connected components of the essential metrizability space~$\essSp\subset\solSp$ are going to be introduced later, and are typically denoted by $(s,u)$.
	\item[Parameters] specify a projective class, particularly such which contains metrics with one, essential projective symmetry. Typically we use the letters $\xi,\lambda$ as well as $C,h,\varepsilon$, consistent with~\cite{matveev_2012,manno_2018}. The different roles that these parameters play are reviewed in Appendix~\ref{app:normal.forms.dom.2}.
\end{enumerate}
The terminology is briefly summarized in Table~\ref{tab:terminology}.

\begin{table}[!ht]
	\begin{center}
		\begin{tabular}{llc}
			\toprule
			\textbf{Data} & \textbf{Space} & \textbf{Notation} \\
			\midrule
			Variables & differentiable manifold $M$ & $x,y$ \\
			(Dini) coordinates & Liouville metrization space $\solSp$ & $u_1,u_2$ \\
			(Distinguished) coordinates & connected component of the space $\essSp\subseteq\solSp$ & $s,u$ \\
			Parameters & (family of) projective classes & $\xi,\lambda,C,h,\varepsilon$ \\
			\bottomrule
		\end{tabular}
		\caption{Terminology used for the specification of points (on a manifold), metrics (within a particular projective class) and to single out a specific projective class.}\label{tab:terminology}
	\end{center}
\end{table}

\begin{proposition}\label{prop:uniqueness}
 The normal forms (C.8) and (C.9) in Table~\ref{tab:mobility2} can, almost everywhere on their domain, be expressed in terms of hypergeometric polynomials.
\end{proposition}

\noindent In dimension~2, metrics with one, essential projective symmetry typically have degree of mobility~2, compare Appendices~\ref{app:normal.forms.dom.2} and~\ref{app:normal.forms.dom.3}. The following theorem, one of the main results of the current paper, describes the geometry of these metrics within their projective classes.

Particularly, the first part of the theorem states how to canonically divide the solution space~$\solSp$ into a number of components and how this division is related to the action of the projective symmetry.
The second part provides distinguished coordinates that can be employed to find the normal forms for metrics with one essential normal forms given in Appendix~\ref{app:normal.forms.dom.2}. This procedure has indeed already been executed, for the case of Section~\ref{sec:proof.distinguished.coordinates.b.3} of the present paper, in~\cite{manno_2018}.
\begin{theorem}\label{thm:distinguished.coordinates}
 Let $g$ be a 2-dimensional metric with degree of mobility 2.
 Let us furthermore assume that $g$ admits exactly one projective vector field $X$ (up to a factor), and that $X$ is essential (i.e., non-homothetic).\smallskip

 \noindent
 a) Let $r$ be the number of real eigenvalues of $\lie_X:\solSp\to\solSp$.
 The essential metrizability space $\essSp\subset\solSp$ decomposes into $2^r$ connected components.
 The connected components are, respectively:
 \begin{itemize}[label={---}]
 	\item If $2\geq r>0$, we have $2^r$ copies of~$\mathbb{R}^2$. Some of them might be diffeomorphic, e.g.\ in the complex Liouville normal forms and the \emph{special cases} of the classification (see Table~\ref{tab:normal.forms.dom.2} in Appendix~\ref{app:normal.forms.dom.2}).
 	\item If no real eigenvalue exists, there is one component that is isomorphic to $\mathbb{S}^1\times\mathbb{R}$.
 \end{itemize}
 \smallskip

 \noindent
 b) There are distinguished coordinates on each connected component of $\essSp$, adjusted to the flow of $X$.
 Let us call~$(s,u)$ the adjusted coordinates such that~$s$ is a parameter along the orbit of the projective flow and such that~$u$ is constant along orbits.
 In terms of Dini coordinates, we find that a solution~$\sigma$ of~\eqref{eqn:linear.system} can be expressed as
 \[
   \sigma = f_1(s,u)\,\sigma_1+f_2(s,u)\,\sigma_2
 \]
 where $(\sigma_1,\sigma_2)$ is a Dini basis as detailed in Table~\ref{tab:dini}. The functions $f_i$ are functions on $\mathbb{R}^2_{(s,u)}$ as in Table~\ref{tab:mobility2}.
\end{theorem}

\begin{table}[!ht]
\begin{center}
\textbf{The distinguished coordinates for degree of mobility 2}\smallskip

\begin{tabular}{cg{2cm}x{2.3cm}g{2cm}x{2.1cm}}
	\toprule
	& (I) & (II) & (III${}_0$) & (III${}_\lambda$) \\
	\toprule
	$r$ (real eigenvalues) & 2 & 1 & 0 & 0 \\
	\toprule
	$\sigma$ (parametrization) &
	\begin{minipage}{2cm}\centering $ue^{\lambda s}\,\sigma_1$ \\ $+e^s\,\sigma_2$ \end{minipage} &
	\begin{minipage}{2.3cm}\centering $e^s\,\sigma_1+$ \\ $e^s\,(s+\ln(u))\,\sigma_2$ \end{minipage} &
	\begin{minipage}{2cm}\centering $u\sin(s)\,\sigma_1$ \\ $+u\cos(s)\,\sigma_2$ \end{minipage} &
	\begin{minipage}{2.1cm}\centering $ue^{\lambda s}(\sin(s)\sigma_1$ \\ $+\cos(s)\,\sigma_2)$ \end{minipage} \\
	\midrule
	$u$ (invariant parameter) &
	\begin{minipage}{2cm}\centering $\frac{|u_1|}{|u_2|^\lambda}$ \\ ($\lambda=\frac{\lambda_2}{\lambda_1}$) \end{minipage} &
	$\frac{1}{|u_1|}\,e^\frac{u_2}{u_1}$ &
	$u_1^2+u_2^2$ &
	\begin{minipage}{2cm}\centering $\frac{\ln(u_1^2+u_2^2)}{2\lambda}$ \\ $-\arctan\left(\frac{u_1}{u_2}\right)$ \end{minipage}
	\\
	\midrule
	$s$ (orbit paramter) &
	$\ln(u_1)$ & $\ln(u_1)$ & $\arctan\left(\frac{u_2}{u_1}\right)$ & $\arctan\left(\frac{u_2}{u_1}\right)$
	\\
	\midrule
	orbits &
	\begin{minipage}{2cm}
	\centering
	\begin{tikzpicture}[thick,scale=0.6, every node/.style={scale=0.6}]
	  \pgfmathsetmacro{\rx}{0.00099}
	  \pgfmathsetmacro{\ry}{0.0004}
	  \pgfmathsetmacro{\kx}{0.4}
	  \pgfmathsetmacro{\ky}{0.5}
	  \draw[->] (0,-0.5) -- (0,1.5) node[above] {$u_2$};
	  \draw[->] (-0.5,0) -- (1.5,0) node[right] {$u_1$};
	  \draw[scale=1,domain=-2500:1100,smooth,variable=\t,black,dashed] plot ({exp(\rx*\t)*\kx},{exp(\ry*\t)*\ky});
	\end{tikzpicture}
	\\
	\begin{tikzpicture}[thick,scale=0.6, every node/.style={scale=0.6}]
	  \pgfmathsetmacro{\rx}{0.00099}
	  \pgfmathsetmacro{\ry}{-0.0004}
	  \pgfmathsetmacro{\kx}{0.4}
	  \pgfmathsetmacro{\ky}{0.5}
	  \draw[->] (0,-0.5) -- (0,1.5) node[above] {$u_2$};
	  \draw[->] (-0.5,0) -- (1.5,0) node[right] {$u_1$};
	  \draw[scale=1,domain=-2500:1100,smooth,variable=\t,black,dashed] plot ({exp(\rx*\t)*\kx},{exp(\ry*\t)*\ky});
	\end{tikzpicture}
	\end{minipage}
	&
	\begin{minipage}{2cm}
	\centering
	\begin{tikzpicture}[thick,scale=0.6, every node/.style={scale=0.6}]
		\pgfmathsetmacro{\kx}{0.5}
		\pgfmathsetmacro{\ky}{0.002}
		\draw[->] (0,-0.5) -- (0,1.5) node[above] {$u_2$};
		\draw[->] (-1,0) -- (1,0) node[right] {$u_1$};
		\draw[scale=1,domain=-10:0.9,smooth,variable=\t,black,dashed] plot ({exp(\t)*\kx},{exp(\t)*\ky+\t*exp(\t)*\kx});
	\end{tikzpicture}
	\end{minipage}
	&
	\begin{minipage}{2cm}
	\centering
	\begin{tikzpicture}[thick,scale=0.6, every node/.style={scale=0.6}]
	  \pgfmathsetmacro{\kx}{0.3}
	  \pgfmathsetmacro{\ky}{0.5}
	  \draw[->] (0,-1) -- (0,1) node[above] {$u_2$};
	  \draw[->] (-1,0) -- (1,0) node[right] {$u_1$};
	  \draw[scale=1,domain=-180:180,smooth,variable=\t,black,dashed] plot ({\kx*cos(\t)-\ky*sin(\t)},{\kx*sin(\t)+\ky*cos(\t)});
	\end{tikzpicture}
	\end{minipage}
	&
	\begin{minipage}{2cm}
	\begin{tikzpicture}[thick,scale=0.6, every node/.style={scale=0.6}]
	  \pgfmathsetmacro{\lam}{0.003}
	  \pgfmathsetmacro{\kx}{0.05}
	  \pgfmathsetmacro{\ky}{0.05}
	  \draw[->] (0,-1) -- (0,1) node[above] {$u_2$};
	  \draw[->] (-1,0) -- (1,0) node[right] {$u_1$};
	  \draw[scale=1,domain=-10:800,smooth,variable=\t,black,dashed] plot ({exp(\lam*\t)*(\kx*cos(\t)-\ky*sin(\t))},{exp(\lam*\t)*(\kx*sin(\t)+\ky*cos(\t))});
	 \end{tikzpicture}
	 \end{minipage} \\
	\bottomrule
\end{tabular}
\end{center}
\caption{The table shows the distinguished coordinates~$(s,u)$ on the Liouville metrization space~$\solSp$, or rather on the essential metrizability space~$\essSp$, in comparison with the standard Dini coordinates~$(u_1,u_2)$ used in~\cite{matveev_2012,manno_2018}. The parameter $\lambda$ characterizes the eigenvalue structure of $\lie_X$: In case (I),~$\lambda$ is the ratio between the two eigenvalues (see~\eqref{eqn:normal.forms.Lw} below). In case (III), $\lambda$ is the ratio between the real and imaginary part of the eigenvalue ($\lambda=0$ implies the eigenvalues are purely imaginary). The last row sketches the orbits of the action of~$X$ in~$\solSp$. In the case (I), the geometric shape of the orbit depends on the sign of $\lambda$ (above: $\lambda>0$, below: $\lambda<0$). For the case (III), there are two columns as the situation for vanishing and non-vanishing $\lambda$ are essentially different.}\label{tab:mobility2}
\end{table}

\noindent The second main result concerns superintegrable systems. Recall from Definition~\ref{def:superintegrability} that, in dimension~2, a superintegrable system requires the existence of 3 functionally independent integrals of motion, including the Hamiltonian itself. Since we know, explicitly in terms of a normal form, the metrics with one, essential projective vector field that have degree of mobility 3, we may ask which of these systems indeed constitute superintegrable metrics.
\begin{theorem}\label{thm:superintegrable.systems}
Let $g$ be a 2-dimensional metric with degree of mobility 3.
Let us furthermore assume that $g$ admits exactly one projective vector field~$X$ (up to a constant factor), and that $X$ is essential (i.e., non-homothetic). Then~$g$ gives rise to free superintegrable systems,
where $g$ corresponds to $\sigma=\sigma[\theta,\varphi]$ and
\begin{subequations}\label{eqn:new.superintegrable.basis}
\begin{align}
 \sigma &= \sin(\theta)\,\cos(\varphi)\,\sigma_1 +\sin(\theta)\,\sin(\varphi)\,\sigma_2 +\cos(\theta)\,\sigma_3 \\
 \bar{\sigma} &= \cos(\theta)\cos(\varphi)\,\sigma_1 +\cos(\theta)\sin(\varphi)\,\sigma_2 -\sin(\theta)\,\sigma_3 \\
 \hat{\sigma} &= -\sin(\varphi)\,\sigma_1+\cos(\varphi)\,\sigma_2
\end{align}
\end{subequations}
where $\sigma_i$ are obtained via Formula~\eqref{eqn:sigma.g} from $g_i$, see~\eqref{eqn:supint.generators} in Appendix~\ref{app:normal.forms.dom.3}.
Indeed, these superintegrable systems are free second order maximally superintegrable systems in the usual sense.
\end{theorem}

\noindent For the well-established notion of superintegrable systems, see reference~\cite{kalnins_2018}, for instance.
Note that we consider only the case of free Hamiltonians here (i.e., no potential term exists), see also~\cite{vollmer_2018} for a subsequent study of the full superintegrable systems.
In Section~\ref{sec:integration.of.connection}, we explicitly compute the trajectories of the geodesics of the (superintegrable) systems described in Theorem~\ref{thm:superintegrable.systems}.

\section{Uniqueness of the normal forms for metrics with one, essential projective symmetry}\label{sec:classification}

This current section has two principal parts. First, in Subsection~\ref{sec:group.actions.symmetries} we review the strategy of~\cite{manno_2018}, compiling a synopsis how symmetries and group actions are exploited to obtain mutually non-diffeomorphic normal forms for metrics with one, essential projective symmetry --- the classification itself is reviewed in Appendices~\ref{app:normal.forms.dom.2} and~\ref{app:normal.forms.dom.3}.
Afterwards, in Subsection~\ref{sec:uniqueness}, we critically review the uniqueness of these normal forms. Indeed, two of the normal forms are given in terms of unevaluated integrals, allowing for an integration constant to be chosen. This weaker form of uniqueness is remedied by an alternative representation in terms of special functions.
Some of the material in this current section is going to be referred to in the proofs of Theorems~\ref{thm:distinguished.coordinates} and~\ref{thm:superintegrable.systems}, too.

\subsection{Group actions and symmetries}\label{sec:group.actions.symmetries}

Let $g$ be a metric and let $\solSp$ be its associated metrization (Liouville) space. On~$\solSp$, or rather on~$\metrSp$, there are two major group actions:
First, the action of the projective group. Second, the action of the isometry group.
Most importantly, any projective vector field $X$ induces a linear mapping on $\solSp$, as~\eqref{eqn:LX.action} shows.
\begin{lemma}[\cite{matveev_2012,manno_2018}]\label{la:normal.forms.Lw}
Let $X$ be a projective vector field. Then there is a 2-dimensional subspace $\solSp'\subseteq\solSp$ that is invariant under the Lie derivative $\lie_X:\solSp'\to\solSp'\subseteq\solSp$. There exists a basis $(\sigma_1,\sigma_2)$ of~$\solSp'$ such that~$\lie_X$ assumes one of the following normal forms up to rescaling~$X$.
\begin{equation}\label{eqn:normal.forms.Lw}
 \text{(I)}\quad \begin{pmatrix} \lambda & 0\\ 0 & 1 \end{pmatrix}\quad\text{with $|\lambda|\geq1$}\,;\qquad\qquad
 \text{(II)}\quad \begin{pmatrix} 1 & 1\\ 0 & 1 \end{pmatrix}\,;\qquad\qquad
 \text{(III)}\quad \begin{pmatrix} \lambda & -1\\ 1 & \lambda \end{pmatrix} \quad\text{with $\lambda\geq0$}.
\end{equation}
\end{lemma}
\medskip

\noindent The action in~$\solSp$ of (the projective algebra generated by) $X$ can also be investigated in terms of a pullback operation.
Let $\phi_t$ be the local flow of $X$, $(\sigma_1,\sigma_2)$ a basis of the Liouville metrization space $\solSp$ according to~\eqref{eqn:normal.forms.Lw}, and $\sum u_{i}\sigma_i\in\solSp$ a linear combination, with $u_{i}\in\mathbb{R}$. We have the following transformation rules.
\begin{subequations}\label{eqn:pullbacks.dom2}
\begin{align}
 (I)\quad\phi^*_t
 \left(\sum u_{i}\sigma_i\right) &= e^{\lambda t}u_1\sigma_1+e^tu_2\sigma_2 \\
 (II)\quad\phi^*_t
 \left(\sum u_{i}\sigma_i\right) &= e^{\lambda t}u_1\sigma_1+(te^tu_1+e^tu_2)\sigma_2 \\
 (III)\quad\phi^*_t
 \left(\sum u_{i}\sigma_i\right) &= e^{\lambda t}\left( (\cos(t)u_1-\sin(t)u_2)\sigma_1+(\sin(t)u_1+\cos(t)u_2)\sigma_2\right)
\end{align}
\end{subequations}

\noindent What about the action of the isometry group? In fact, if two projective classes are isometric, they are identical. This fact has been exploited in~\cite{matveev_2012,manno_2018} to achieve the parameter restrictions in~\eqref{eqn:normal.forms.Lw}. Similarly, on the level of the underlying manifold, a swapping of the variables $x\leftrightarrow y$ permits one to reduce the normal forms~\cite{manno_2018}.

Before we continue further, let us first investigate the parameter $\lambda$ more closely, justifying the assumption $\lambda\geq0$ in~\eqref{eqn:normal.forms.Lw}.
\begin{lemma}[\cite{manno_2018}]
	Consider two projectively equivalent metrics of label (III) and identical type A, B or C in Table~\ref{tab:normal.forms.dom.2}, with parameters $\lambda$, $\lambda'$.
	The parameters satisfy $\lambda'=\pm\lambda$.
\end{lemma}
\begin{proof}
	The projective class is fixed by the eigenvalues on the 2-dimensional space $\solSp$ (the space is 2-dimensional due to Theorem~2 of \cite{manno_2018}).
	Up to a constant factor, we have
	$
	\lie_X \begin{psmallmatrix} \sigma_1 \\ \sigma_2 \end{psmallmatrix}
	=\begin{psmallmatrix} \lambda & -1 \\ 1 & \lambda \end{psmallmatrix}
	\begin{psmallmatrix} \sigma_1 \\ \sigma_2 \end{psmallmatrix}\,,
	$
	and this means the Lie derivative, as an endomorphism on $\solSp$, has the two complex eigenvalues
	$\left\{ \lambda+i\,, \lambda-i\right\}$.
	The only further freedom this allows for, is to replace $\lambda\to-\lambda$, which does not change the set of eigenvalues (any other multiplication by a constant would; however the replacement reorders the pair of eigenvalues).
	Therefore, $\lambda\geq0$ is the optimal parameter range, in the  sense that for different choices of such $\lambda$, the projective classes are different.
\end{proof}

\noindent Let us now discuss the general strategy of~\cite{manno_2018}, which might be described as follows: Consider the space~$\solSp$ of solutions to~\eqref{eqn:linear.system.2D}. For metrics with one, essential projective symmetry a general description (however not sharp) of such spaces is given in~\cite{matveev_2012}, in terms of a number of parameters. Use the action of the isometry group to reduce the number and range of the parameters.
Thus, we find a minimal description of the spaces~$\solSp$. Take one of these spaces, denoted $\solSp$ from now on. The action of the projective group, via its pullback $\phi^*_t$, provides us with a simple isometry that we can use to restrict to some representative for each orbit of the action of the projective symmetry~$X$. In a final step, it needs to be checked if two given representatives are, indeed, isometric.
This outlined strategy is facilitated by a number of simple invariant properties and objects under isometries and the projective group action. Specifically, the available invariants include:
\begin{enumerate}[font=\emph,leftmargin=0pt,align=left,itemindent=*,labelindent=0\parindent,label=\alph*)]
  \item \emph{Degree of Mobility.}
  The degree of mobility is invariant under projective transformation, as is the eigenvalue configuration of~$\lie_X$ on the space~$\solSp$.
  \item \emph{Action of projective symmetries.}
  Metrics on the orbit of the projective symmetry are isometric (diffeomorphic) and thus we can choose a representative on each orbit and ignore the other metrics on this orbit.
  Moreover, the set of poles and zeros of the (square of the) length of the projective vector field is preserved under diffeomorphisms, and this can be exploited to identify isometries (similarly to Benenti tensors mentioned below).
  Finally, the action of the projective symmetry on the space $\solSp$ is an important tool, as eigenspaces of~$\lie_X$ are geometric objects and thus preserved under diffeomorphisms.
  \item \emph{Benenti tensors.}
  Benenti tensors~\eqref{eqn:benenti.tensor}, with $g$ admitting the homothetic vector field~$X$, transform linearly under diffeomorphisms, as eigenspaces of $\lie_X$ are preserved. This permits to derive candidates for local isometries, and thus Benenti tensors can be used as a means to find all diffeomorphisms mediating between candidates of normal forms. Benenti tensors are also important in the study of degenerate solutions to~\eqref{eqn:linear.system}, particularly in higher dimensions~\cite{manno_2018ben}.
  \item \emph{Swapping of variables.} There is an additional freedom inherently present in some of the normal forms. It typically manifests itself, in~\cite{matveev_2012} and~\cite{manno_2018}, as a freedom to exchange the roles of the two variables. This additional freedom has at least two effects on the normal forms: On the one hand, it permits to identify identical projective classes among the description in~\cite{matveev_2012}. On the other hand, it reveals \emph{special cases}, see Table~\ref{tab:mobility2}, i.e.\ the existence of additional isometric freedom that allows one to restrict the ranges of some of the parameters in the normal forms (refer to Appendix~\ref{app:normal.forms.dom.2} for details).
\end{enumerate}
This permits one to identify the possible diffeomorphisms that connect the explicit solutions of~\eqref{eqn:linear.system.2D}. In fact, asking for the above invariant properties to be observed, we can restrict ourself to a small set of diffeomorphisms. As a second step, one can check whether these diffeomorphisms indeed map one solution onto another (this might produce extra restrictions on the parameters).
The result is the list of mutually non-diffeomorphic (locally non-isometric) normal forms, see Appendices~\ref{app:normal.forms.dom.2} and~\ref{app:normal.forms.dom.3}.
Let us now investigate the uniqueness of these normal forms more carefully.

\subsection{Uniqueness of the Normal Forms: Proof of Proposition~\ref{prop:uniqueness}}\label{sec:uniqueness}
The normal forms in \cite{manno_2018} are unique in the following sense: After specifying the antiderivative of unevaluated integrals, if present, there does not exist a change of variables such that the resulting expression for the metric is identical in form with the initial one, but with different parameter values. In other words, we may say that the parameters are sharp, not allowing for variable transformations between the normal forms.
However, antiderivatives have to be specified first. Consider the normal forms (C.8) and (C.9) in Appendix~\ref{app:normal.forms.dom.2}. They indeed allow for a slight arbitrariness, as they involve unevaluated integrals in one variable, denoted by $Y(y)$ and $Y_\lambda(y)$ respectively. We are therefore free to choose the integration constant, but this supposed arbitrariness is naturally absorbed by a translational freedom in the other variable, called~$x$ in Table~\ref{tab:mobility2}. Thus, while (C.8) and (C.9) are unique in the stated sense, there is also some arbitrariness present.

The current subsection is devoted to a discussion of the uniqueness of these normal forms\footnote{The discussion is based on the previous versions of the preprint arXiv:1705.06630 of~\cite{manno_2018}, particularly Lemmas~\ref{la:erfi.representation} and~\ref{la:hypergeom.C9} correspond to Remarks~10 and~13 of its versions~1 and~2. These were removed from later versions due to length considerations.}.
Can we remove the mentioned remaining arbitrariness from the normal forms?
The answer, in fact, is yes (almost everywhere) as the following two lemmas show.
\begin{lemma}\label{la:erfi.representation}
Metric (C.8) is isometric to
\begin{equation*}
	g = \kappa\,\bigg( x+\erfi(\nicefrac{1}{\sqrt{y}}) \bigg)\,dxdy\,,
	\quad\text{if $y<0$}\,,
	\qquad\text{and}\qquad
	g = \kappa\,\bigg( x+\erf(\nicefrac{1}{\sqrt{y}}) \bigg)\,dxdy\,,
	\quad\text{if $y>0$}\,.
\end{equation*}
\end{lemma}
\begin{proof}
	In Section~3.1.3.\ of \cite{matveev_2012}, the function $Y$ is obtained as the derivative of a function $Y_1$, where $Y=Y_1'$ satisfies the ODE
	$y^2\,Y_1''-\frac12 (y-3) Y_1'+\frac12 Y_1 =0$.
	This equation is similar to Kummer's equation. Indeed we can write the solution in terms of the confluent hypergeometric function ${}_1F_1$,
	\begin{equation*}
	Y_1(y)
	= C_1\left(
	2\sqrt{6}y\,(y-3)\,\,{}_1F_1\left(\tfrac12,\tfrac32,\tfrac{3}{2y}\right)
	+6\sqrt{y}e^{\nicefrac{3}{2y}}
	\right)+C_2(y-3)
	\quad\text{and w.l.o.g.}\quad
	Y(y) = y\,\,{}_1F_1\left(\tfrac12,\tfrac32,\tfrac{3}{2y}\right)\,.
	\end{equation*}
	For $y\ne0$, the special function ${}_1F_1\left(\frac12,\frac32,\bullet\right)$ can be represented by the (imaginary) error function $\erf$ ($\erfi$),
	\begin{align*}
	Y_1(y)
	&= C_1\left(
	\sqrt{6\pi}\,(y-3)\erfi\left(\frac{\sqrt{6}}{2\sqrt{y}}\right)
	+6\sqrt{y}e^{\nicefrac{3}{2y}}
	\right)+C_2(y-3),
	&\text{for $y<0$} \\
	Y_1(y)
	&= C_1\left(
	\sqrt{6\pi}\,(y-3)\erf\left(\frac{\sqrt{6}}{2\sqrt{y}}\right)
	+6\sqrt{y}e^{\nicefrac{3}{2y}}
	\right)+C_2(y-3),
	&\text{for $y>0$}\,.	
	\end{align*}
	Therefore, w.l.o.g.\ one may assume
	$Y=\erfi(\nicefrac{1}{\sqrt{y}})$ for $y<0$
	and
	$Y=\erf(\nicefrac{1}{\sqrt{y}})$ for $y>0$.
\end{proof}
\noindent For the normal forms (C.9), the description via hypergeometric functions is possible analogously, although it is much less concise.

\begin{lemma}\label{la:hypergeom.C9}
Metrics (C.9) are isometric to $\kappa\,(\Upsilon_\lambda(y)+x)\,dxdy$ where $\Upsilon_\lambda$ is a hypergeometric polynomial.
\end{lemma}
\begin{proof}
Analogously to Lemma~\ref{la:erfi.representation}, $\Upsilon_\lambda(y)=\Xi_\lambda'(y)$ is obtained from an ODE, given in~\cite{matveev_2012},
\begin{equation*}
  (y^2+1)\,\Xi_\lambda''-\frac12 (y-3\lambda) \Xi_\lambda'+\frac12 \Xi_\lambda =0\,.
\end{equation*}
We find the solution in terms of a generalized hypergeometric function, denoted by $\hypergeom$,
\[
	\Xi_\lambda(y) = c_1\,(3\lambda-y)
	 +c_2\,(y+i)^{-\frac{i}{4}\,(3\lambda+5i)}\,
	 \hypergeom\left( \left[-\tfrac{3i}{4}\,(\lambda+i)\right],
			  \left[-\tfrac{3i}{4}\,(\lambda+3i)\right],
			  -\tfrac{i}{2}\,(y+i)
		    \right).
\]
Thus, $\Upsilon_\lambda$ is a hypergeometric polynomial also.
\end{proof}

\noindent Lemmas~\ref{la:erfi.representation} and~\ref{la:hypergeom.C9}, combined, prove Proposition~\ref{prop:uniqueness}.

\section{Projective vector fields and first integrals}\label{sec:projective.vector.fields.integrals}

In this section, we review projective connections employing both an ODE-jet perspective and classical-mechanic integrals of motion. Particularly, we shall see how homogeneous polynomial integrals of the geodesic equation (i.e.\ Killing tensors) are interrelated with rational integrals of its associated projective connection.

\subsection{ODE-Jets perspective: projective connections as quotient systems}\label{sec:ode.perspecitve}

A geometric viewpoint for obtaining system \eqref{eq:proj.conn.gen.multidim} is provided by the jet spaces. In fact, system \eqref{eq:param.geod.equ} can be interpreted as an $N$-codimensional submanifold of the second jet space $J^2\pi$ of the trivial bundle $\pi:\mathbb{R}\times M\to\mathbb{R}$.
Recall that a point of the jet space $J^2\pi$ is an equivalence class $[\gamma]^2_t$ of parametrized curves whose Taylor expansion in $t$ is the same as that of $\gamma$ in $t$ up to the second order.
To any curve $\gamma:I\subseteq\mathbb{R}\to M$ one can associate a (local) section of $\pi$, and vice-versa. From now on we consider only local sections $\gamma$ of~$\pi$ that come from regular curves.
We denote by $(t,y^i,\dot{y}^i,\ddot{y}^i)$ a local chart of $J^2\pi$.
On the other hand, system~\eqref{eq:proj.conn.gen.multidim} can be interpreted as an $(N-1)$-codimensional submanifold of the second jet space $J^2(M,1)$ of $1$-dimensional submanifolds of $M$, with $(x,y^i_x,y^i_{xx})$ being a local chart of $J^2(M,1)$.
In fact, similarly to the case of $J^2\pi$, a point of $J^2(M,1)$ is an equivalence class $[\alpha]^2_p$ of $1$-dimensional submanifolds having with the $1$-dimensional submanifold $\alpha$ a contact of order $2$ in $p\in\alpha$, i.e.\ if they are locally described by graphs of functions then the Taylor expansions of these functions in the given point $p$ coincide up to order~2. The link between systems~\eqref{eq:param.geod.equ} and~\eqref{eq:proj.conn.gen.multidim} is provided by the following natural map:
\begin{equation}\label{eq:natural.map}
[\gamma]^2_t\in J^2\pi\to [\gamma(J)]_{\gamma(t)}^2\in J^2(M,1)
\end{equation}
where $J\subseteq I$ is a suitable small subinterval of $I$ such that $\gamma(J)$ is a $1$-dimensional submanifold of $M$. The local expression of \eqref{eq:natural.map} is
\begin{equation}\label{eq:map:J2}
(t,x,y^i,\dot{x},\dot{y}^i,\ddot{x},\ddot{y}^i)\in J^2\pi \to \left(x,y^i,\frac{\dot{y}^i}{\dot{x}},\frac{\ddot{y}^i\dot{x}-\ddot{x}\dot{y}^i}{\dot{x}^3}\right) = (x,y^i,y^i_{x},y^i_{xx})\in J^2(M,1)
\end{equation}
\begin{proposition}\label{prop:proj}
	System \eqref{eq:param.geod.equ} projects to system \eqref{eq:proj.conn.gen.multidim} via the map \eqref{eq:natural.map}.
\end{proposition}

\noindent So, in view of Proposition~\ref{prop:proj}, we can interpret the projective connection~\eqref{eq:proj.conn.gen.multidim} (associated to the connection $\nabla$) as the quotient equation of~\eqref{eq:param.geod.equ} via the map~\eqref{eq:natural.map}.

From another point of view, we can interpret system~\eqref{eq:param.geod.equ} as the following $1$-dimensional distribution $J^1\pi=\mathbb{R}\times TM$:
\begin{equation}\label{eq:Dt}
\langle \overline{D}_t \rangle = \langle \partial_t + \dot{y}^k\partial_{y^k} - \Gamma^k_{ij}\dot{y}^i\dot{y}^j \partial_{\dot{y}^k}\rangle
\end{equation}
where $\overline{D}_t$ denotes the restriction of the total derivative operator $D_t:= \partial_t + \dot{y}^k\partial_{y^k} +\ddot{y}^k\partial_{\dot{y}^k}$ to system \eqref{eq:param.geod.equ}. The vector field $\overline{D}_t$ is called the \emph{spray} or the \emph{dynamical vector field} associated to system \eqref{eq:param.geod.equ}.
Analogously, we can view system~\eqref{eq:proj.conn.gen.multidim} as the following $1$-dimensional distribution on $J^1(M,1)=\mathbb{P}TM$
\begin{equation}\label{eq:Dx}
\langle \overline{D}_x \rangle
=\bigg\langle \partial_x + y^i_x\partial_{y^i}
-\bigg(
\Gamma_{11}^k
+ (2\Gamma_{1m}^k-\delta^k_m\Gamma_{11}^1)  y^m_x
+ (\Gamma_{jm}^k  - 2\delta^k_m \Gamma_{1j}^1 ) y^j_x y^m_x
- \Gamma_{ij}^1 \, y^i_x y^j_x y^k_x
\bigg) \partial_{y^k_x}
\bigg\rangle
\end{equation}
where $\overline{D}_x$ is the restriction of the total derivative $D_x:=\partial_x + y^i_x\partial_{y^i}+y^i_{xx}\partial_{y^i_x}$ to system  \eqref{eq:proj.conn.gen.multidim}.
\smallskip

\noindent A straightforward computation shows that the distribution~\eqref{eq:Dt} projects to the distribution~\eqref{eq:Dx} via the map~\eqref{eq:natural.map}. Thus, system~\eqref{eq:param.geod.equ} is a covering of the system~\eqref{eq:proj.conn.gen.multidim}, via the map~\eqref{eq:natural.map}, in the sense of Krasil'shchik-Vinogradov~\cite{Krasil}.

\subsection{First integrals of the geodesic equation and of its projective connection}
A first integral of \eqref{eq:param.geod.equ} is a function $F(t,y^i,\dot{y}^i)$ such that
$\overline{D}_t(F)=0$
(see \eqref{eq:Dt} for the definition of $\overline{D}_t$).
Let us assume, from now on, that~$\nabla$ is a Levi-Civita connection of a metric~$g$.
Therefore system~\eqref{eq:param.geod.equ} is equivalent to the Hamiltonian system with Hamiltonian function $H=\frac12 g^{ij}p_ip_j$, where~$p_i$ are the momenta. In this case the above criterion for~$F$ to be a first integral of ~\eqref{eq:param.geod.equ}  coincides with the vanishing of the Poisson bracket $\{H,F\}=0$, once replacing, in $F$,~$\dot{y}^i$ by~$g^{ij}p_j$.
Analogously, a first integral of \eqref{eq:proj.conn.gen.multidim} is a function $f(x,y^i,y^i_x)$ such that
$\overline{D}_x(f)=0$ (see \eqref{eq:Dx} for the definition of~$\overline{D}_x$).
Below we shall see how to associate an integral of~\eqref{eq:proj.conn.gen.multidim} starting from an integral of~\eqref{eq:param.geod.equ}.
This will also provide a way to construct an integral of~\eqref{eq:param.geod.equ} (and, consequently, an integral of~\eqref{eq:proj.conn.gen.multidim}) starting from a projective vector field.

\begin{proposition}\label{prop:f.i}
	If $F\left(x,y^i,\frac{\dot{y}^i}{\dot{x}}\right)$ is a first integral of system \eqref{eq:param.geod.equ}, then $F(x,y^i,y^i_x)$ is a first integral of system~\eqref{eq:proj.conn.gen.multidim}.
\end{proposition}
\begin{proof}
	This follows by a straightforward computation.
\end{proof}

\begin{corollary}\label{cor:int.proj.conn}
	If system \eqref{eq:param.geod.equ} admits a time-independent homogeneous quadratic integral~$h_{ij}\dot{y}^i\dot{y}^j$, in velocities, then system \eqref{eq:proj.conn.gen.multidim} admits the following first integral:
	\begin{equation}\label{eq:rational.int}
	\frac{h_{11}+2h_{1m}y^m_x+h_{mn}y^m_xy^n_x}{g_{11}+2g_{1m}y^m_x+g_{mn}y^m_xy^n_x}\,.
	\end{equation}
\end{corollary}
\begin{proof}
	Since $h_{ij}\dot{y}^i\dot{y}^j$ is a first integral of~\eqref{eq:param.geod.equ}, also
	$$
	F=\frac{h_{ij}\dot{y}^i\dot{y}^j}{g_{ij}\dot{y}^i\dot{y}^j}
	$$
	is a first integral of system~\eqref{eq:param.geod.equ} as such a system always admits the polynomial~$g_{ij}\dot{y}^i\dot{y}^j$ as a first integral\footnote{In fact, if $f$ and $g$ are integrals of \eqref{eq:param.geod.equ}, then $\overline{D}_t\left(\frac{f}{g}\right)=\frac{\overline{D}_t(f)g-\overline{D}_t(g)f}{g^2}=0$, so that $\frac{f}{g}$ is an integral of~\eqref{eq:param.geod.equ}.}. It is of the form \smash{$F=F\left(x,y^i,\frac{\dot{y}^i}{\dot{x}}\right)$} as it is the quotient of two homogeneous polynomials. So, in view of Proposition~\ref{prop:f.i}, $F(x,y^i,y^i_x)$ is a first integral of~\eqref{eq:proj.conn.gen.multidim}, i.e., the function~\eqref{eq:rational.int}
	is the integral we were looking for.
\end{proof}
\begin{remark}
	It is natural to consider as integrals of \eqref{eq:proj.conn.gen.multidim} functions that are the ratio of two polynomials in the variables $y^m_x$ of the same degree as these functions form a closed class w.r.t. point transformations. In fact, the prolongation of such transformations to $J^1(M,1)=\mathbb{P}TM$ is the ratio of two polynomials of first degree in $y^m_x$. More precisely the prolongation to $J^1(M,1)$ of the point transformation
	$\big(x,y^i\big)\to \big(u(x,y^i),v^k(x,y^i)\big)$  is
	$$
	(x,y^i,y^i_x)\to\left(  u,v^j, \frac{v^j_x+v^j_{y^i}y^i_x}{u_x+u_{y^i}y^i_x} \right)
	$$
	and the aforementioned class of functions is invariant up to the above transformations.
\end{remark}


\begin{proposition}[\cite{Hiramatu,Connor_Prince}]\label{prop:right}
	If $X$ is a projective vector field of a $n$-dimensional metric $g$, then
	\begin{equation}\label{eq:first.integ.associated.to.proj.field.2}
	I_X=\left(X_{(i;j)}-\frac{2}{n+1}\mathrm{div}X\, g_{ij}\right)\dot{x}^i\dot{x}^j
	\end{equation}
	is a quadratic first integral of the geodesic flow of $g$.
	The following notation has been used:
	\begin{align*}
	X_{(i;j)} &:= \frac12(X_{i;j}+X_{j;i})
	=\frac12\left(X^a_{,j}\,g_{ai}+X^ag_{ai,j} + X^a_{,i}\,g_{aj}+X^ag_{aj,i} -2X^ag_{ak}\Gamma^k_{ji} \right)
	\\
	\mathrm{div}X &:= X^i_{;i} =X^i_{,i} + X^j\Gamma^i_{ij}
	=\frac{1}{\sqrt{\det g}}(\sqrt{\det g}\,X^i)_{,i}
	\end{align*}
\end{proposition}

\begin{corollary}
	Let $X$ be a projective vector field of a metric $g$. Let
	\begin{equation}\label{eq:hij}
	h_{ij}:=X_{(i;j)}-\frac{2}{n+1}\mathrm{div}X\, g_{ij}\,.
	\end{equation}
	Inserting~\eqref{eq:hij} into~\eqref{eq:rational.int}, we obtain a first integral of the projective connection associated to~$g$.
\end{corollary}
\begin{proof}
	In view of Proposition \ref{prop:right}, $h_{ij}\dot{x}^i\dot{x}^j$ is a time-independent first integral of the geodesic flow of $g$. Then, in view of Corollary \ref{cor:int.proj.conn}, function \eqref{eq:rational.int} with $h_{ij}$ defined by \eqref{eq:hij} is an integral of the corresponding projective connection associated to the Levi-Civita connection of $g$.
\end{proof}

\begin{remark}\label{rem:lagrangian}
	The classical equation of geodesics is the Euler-Lagrange equation of the Lagrangian density  $\sqrt{g_{ij}\dot{x}^i\dot{x}^j}dt$, which is the ``horizontalization''\footnote{Let $\alpha=adt+b_idy^i$ be a $1$-form on $J^1\pi=\mathbb{R}\times TM$ with values in $T^*J^0\pi=T^*(\mathbb{R}\times M)$. The horizontalization of $\alpha$ is the differential form $(a+b_i\dot{y}^i)dt$, that is a differential $1$-form on $J^1\pi$ with values in $T^*\mathbb{R}$. For more details and generalization of this procedure see \cite{Krasil}.} of the pull-back of
	\begin{equation}\label{eq:lagr.dens}
	L = \sqrt{g_{11}+2g_{1i}y^i_x+g_{ij}y^i_xy^j_x}\,dx\,=: \mathsf{L}\,dx\,.
	\end{equation}
	via the map \eqref{eq:map:J2}, restricted to first jet spaces. The projective connection is the Euler-Lagrange equation of the Lagrangian density \eqref{eq:lagr.dens}.
	
	A \emph{divergence} symmetry is a vector field $X=X^i\partial_{y^i}$, where $y^1:=x$, such that the Lie derivative preserves~\eqref{eq:lagr.dens} up to a total divergence, i.e., there exists a function $B=B(y^1,\dots,y^n)$ such that
	\begin{equation}\label{eq:divergence.simm}
	\lie_X(\mathsf{L}\,dx)=\big(X^{(1)}(\mathsf{L})+\mathsf{L}D_x(X^1)\big)dx=D_x(B)dx\,\,
	\footnote{If B=0 then such symmetries are called \emph{variational}. Variational symmetries of the Lagrangian density \eqref{eq:lagr.dens} are the infinitesimal isometries of $g$.}
	\end{equation}
	where $X^{(1)}=X^i\partial_{y^i} + \sum_{i=2}^n\big(D_x(X^i)-y^i_xD_x(X^1)\big)\partial_{y^i_x}$ is the prolongation of $X$ on $J^1(M,1)$ and $D_x$ is the total derivative truncated to the first order. The Noether Theorem says that to any divergence symmetry of~$\mathsf{L}\,dx$ one can associate a first integral of the corresponding Euler-Lagrange equation (see Section 4.4 of \cite{Olver} for a proof). A direct computation (see the discussion of Example~\ref{ex:sphere} below) shows that, in general, projective symmetries are not of divergence type, so the first integral associated to a projective symmetry cannot be computed by the Noether Theorem.
\end{remark}

\begin{example}\label{ex:sphere}
	Let us consider the round metric on the sphere,
	\begin{equation}\label{eq:metric.sphere}
	ds^2=\sin(y)^2dx^2+dy^2\,.
	\end{equation}
	Then the associated projective connection is
	\begin{equation}\label{eq:proj.conn.sphere}
	y_{xx} = \sin(y)\cos(y)+2\cot(y)y_x^2
	\end{equation}
	The vector field $X=\sin(y)^2\cos(x)\partial_y$ is a projective vector field and, by applying \eqref{eq:first.integ.associated.to.proj.field.2}, we obtain that
	\[
	K_X = -2\sin(y)^3\cos(x)\cos(y)dx^2-2\sin(y)^2\sin(x)dxdy
	\]
	is a Killing tensor, implying that
	\[
	F=\sin(y)^3\cos(x)\cos(y)\dot{x}^2+\sin(y)^2\sin(x)\dot{x}\dot{y}
	\]
	is a first integral of the geodesic flow of \eqref{eq:metric.sphere}.
	Furthermore, in view of Corollary \ref{cor:int.proj.conn}, the function
	$$
	f=\frac{\sin(y)^3\cos(x)\cos(y)+\sin(y)^2\sin(x)y_x}{\sin(y)^2+y_x^2}
	$$
	is a first integral of \eqref{eq:proj.conn.sphere}. Finally, we note that $X$ is not a divergence symmetry. In fact, in this case, \eqref{eq:divergence.simm} leads to a polynomial in $y_x$ of fourth degree: the vanishing of the coefficients of such polynomial lead to an incompatible first-order system of PDE (in the unknown function~$B$).
\end{example}

\subsection{Quadratic integrals and projectively equivalent metrics}\label{sec:integrals.projective.geometry}

There is a well-known interdependence of first integrals that are quadratic polynomials in the momenta (resp., in the velocities) and projectively equivalent metrics. To this end, let us recall the correspondence between integrals polynomial in momenta and Killing tensors~\eqref{eqn:killing.tensor}, as discussed in Section~\ref{sec:connections.and.benenti.tensors}.
In turn, the Killing tensors, and thus the associated polynomial integrals, are in correspondence with projectively equivalent metrics.
\begin{theorem}[\cite{topalov_2003}]\label{th:topalov}
	Let $g$ and $\pe{g}$ be metrics on an $n$-dimensional manifold. If they are projectively equivalent then the tensor
	\begin{equation}\label{eqn:integral.g}
	K:=\left(\frac{\det(g)}{\det(\pe{g})}\right)^{\frac{2}{n+1}}\pe{g}\,,
	\end{equation}
	is a Killing tensor of the geodesic flow of $g$. In dimension $n=2$, the above implication is actually an equivalence. In this case, if $K$ is a Killing tensor of order two of a 2-dimensional metric $g$, then
	\begin{equation}\label{eq:metric.from.int}
	\pe{g}=\left(\frac{\det(g)}{\det(K)}\right)^2\,K
	\end{equation}
	is projectively equivalent to $g$.
\end{theorem}

\noindent So, at least in dimension~2, the projective equivalence of two metrics and their integrability are intimately related.
Let~$g$ and~$\pe{g}$ be non-proportional $2$-dimensional projectively equivalent metrics, then the metric $g$ gives rise to an integrable system~$(H,I)$, where~$H=\frac12\,g_{ij}\xi^i\xi^j$ is the Hamiltonian rewritten in terms of velocities. Another integral is given by contracting the Killing tensor~\eqref{eqn:integral.g} with the velocities;
in terms of the Benenti tensor, $L=L(g,\pe{g})$, this integral is obtained as~\cite{topalov_2003}
\begin{equation*}
I(\xi) = \det(L)\,g\big(L^{-1}(\xi),\xi\big)\,.
\end{equation*}
Similarly, an integrable system can be defined for the metric $\pe{g}$.
\medskip

\noindent With Lemma~\ref{la:L.X.nonproportional} in mind, let us pose the following question: Given a metric~$g$ with degree of mobility~$\mathcal{D}$, admitting one, essential projective symmetry $X$, can we reconstruct all projectively equivalent metrics without solving~\eqref{eqn:linear.system.2D}?
We can actually gather the answer from~\cite{manno_2018,matveev_2012}, see also Lemma~\ref{la:L.X.nonproportional}.
\begin{proposition}[\cite{manno_2018,matveev_2012}]\label{prop:IX.independence}

\

\noindent (a) Let $g$ be a $2$-dimensional manifold that admits (up to multiplication by a constant) exactly one projective vector field $X$. We assume that $X$ is not Killing and denote the Hamiltonian by $H=\frac12g^{ij}p_ip_j$.
Then the degree of mobility~$\mathcal{D}$ is either 1, 2 or 3.

\noindent (b) The Lie derivative $\lie_Xg$ allows to reconstruct the full space~$\solSp$ of solutions to~\eqref{eqn:linear.system.2D}, at least in generic cases when~$X$ is essential for~$g$.

\begin{description}[noitemsep,align=left,leftmargin=*,font=\underline]
\item[$\mathcal{D}=1$] In this case, the projective class~$\projcl(g)$ of $g$ is given by $\{kg,k\ne0\}$, and thus $X$ is a homothety for any $g\in\projcl(g)$. The quadratic integral~$I_X$ associated to $X$ is a multiple of the Hamiltonian $H$ and $\lie_Xg$ is proportional to~$g$.
\item[$\mathcal{D}=2$] The space of quadratic integrals is spanned by $H$ and $I_X$ if $X$ is essential.
\item[$\mathcal{D}=3$] If $X$ is not homothetic for $g$, $H$ and $I_X$ span a 2-dimensional space. If this is not a $\lie_X$-invariant subspace of~$\solSp$, then we can define $J_X$ from a second Lie derivative of $g$, i.e.\ from~$\lie_X(\lie_Xg)$. In the first case,~$H$ and~$I_X$ span only the 2-dimensional invariant subspace. In the latter case, $H, I_X$ and $J_X$ span $\solSp$.
\end{description}
\end{proposition}
\noindent Note that in case of a Killing vector field $X$, we are able to choose coordinates $(x,y)$ such that $g=f(y)(dx^2+dy^2)$ and $X=\partial_x$.
A straightforward consequence of Proposition~\ref{prop:IX.independence} is the following observation:
If a 2-dimensional metric $g$ admits an essential projective vector field, then there exists a metric in $\projcl(g)$ non homothetic to~$g$.

\subsection{Integration of the superintegrable projective connection by rational integrals}\label{sec:integration.of.connection}
In the case of $2$-dimensional metrics, when we have two different integrals of the geodesic flow, say~$I_1$ and~$I_2$, then in view of Corollary~\ref{cor:int.proj.conn} we have also two integrals of the associated projective connection, say~$\tilde{I}_1$ and~$\tilde{I}_2$.
We have the system $\tilde{I}_i(x,y,y_x)=c_i$ and, theoretically, we can solve it w.r.t.\ $y$ and $y_x$, i.e., we can find, almost algebraically, all solutions of the projective connection associated to the geodesic flow. This is the case of superintegrable metrics, i.e.\ metrics~\eqref{eqn:supint.generators} of Appendix~\ref{app:normal.forms.dom.3}.
For uniformity of notation, we apply the changing of variables $x\leftrightarrow y$, so that the superintegrable metric~\eqref{eqn:y2.plus.x} becomes $(y+x^2)dxdy$ (and analogously for the metrics projectively equivalent to it). The projective connection associated to this projective class is
	\begin{equation}\label{eq:prj.conn.super}
	y_{xx} = 2\frac{x}{x^2+y}y_x-\frac{1}{x^2+y}y_x^2\,.
	\end{equation}
In view of Theorem~\ref{th:topalov}, metric $(y+x^2)dxdy$ admits two additional quadratic integrals coming from metrics~\eqref{eqn:y2.plus.x.2} and~\eqref{eqn:y2.plus.x.3} (provided we exchange  $x\leftrightarrow y$). This implies, in view of Corollary~\ref{cor:int.proj.conn}, that the projective connection~\eqref{eq:prj.conn.super} admits the following (rational) integrals:
\begin{equation}\label{eqn:integrals.quotient}
	\tilde{I}_1=\frac{2xy_x-(y+x^2)}{y_x}\,,\quad  \tilde{I}_2=\frac{9(y+x^2)y_x^2-4x(x^2+9y)y_x+12y(x^2 + y)}{y_x}\,,\quad y_x\neq 0\,.
\end{equation}
Also, we can solve the system $\tilde{I}_1-\tilde{c}_1=0=\tilde{I}_2-\tilde{c}_2$, $\tilde{c}_i\in\mathbb{R}$, to obtain
\begin{equation}\label{eqn:algebraic.quotient.solution}
	x^4 +4\tilde{c}_1x^3  -6x^2y -12\tilde{c}_1xy +9y^2  -2\tilde{c}_2x +12\tilde{c}_1^2y  +\tilde{c}_1\tilde{c}_2  =0\,.
\end{equation}
Thus, Equation~\eqref{eqn:algebraic.quotient.solution} yields the unparametrized geodesics or, in other words, the trajectories of the (parametrized) geodesic curves of any superintegrable metric, i.e. of any metric in the projective class given by metrics \eqref{eqn:y2.plus.x}--\eqref{eqn:y2.plus.x.3}. Any metric in this projective class is given, up to a diffeomorphic transformation, by Formula~\eqref{eqn:parametrization.dom.3}.
From the equation~$\tilde{I}_1=\tilde{c}_1$, we obtain that unparametrized geodesics are given by the ODE
\begin{equation}\label{eqn:yx.quotient.solution}
	y_x = \frac{x^2+y}{2x-\tilde{c}_1}\,.
\end{equation}
It is straightforwardly verified that any $y=y(x)$ that satisfies~\eqref{eqn:algebraic.quotient.solution} also solves~\eqref{eqn:yx.quotient.solution}, meaning that these two conditions are (differentially) dependent. Wrapping up, we have therefore reduced the integration of the system to solving the algebraic equation~\eqref{eqn:algebraic.quotient.solution}.

Let us now investigate how the proper parametrization for geodesics is obtained, given~\eqref{eqn:algebraic.quotient.solution} and~\eqref{eqn:yx.quotient.solution}, by picking a particular metric within the projective class, i.e.\ specifying the Hamiltonian. Denote the derivative with respect to the coordinate $x$ by $'$, and the derivative w.r.t.\ the proper time by a dot. The transformation between these parametrizations is given by $\dot\gamma=\dot{x}\gamma'$.
The integrals of motion thus are
\[
	H = g(\dot\gamma,\dot\gamma)\,,\qquad
	I_i = I_i(\dot\gamma,\dot\gamma)\,.
\]
We may also express the integrals~\eqref{eqn:integrals.quotient} on the quotient in these terms,
\begin{equation}\label{eqn:quotient.integral}
  \tilde{I}_i = \frac{I_i(\dot\gamma,\dot\gamma)}{g(\dot\gamma,\dot\gamma)} = \tilde{c}_i\in\mathbb{R}\,.
\end{equation}
For proper geodesics, we require the stronger condition that $H(\dot\gamma,\dot\gamma)=k$ and $I_i(\dot\gamma,\dot\gamma)=c_i$, where $k,c_i\in\mathbb{R}$.
Using~\eqref{eqn:quotient.integral} and the transformation rule $\dot\gamma=\dot{x}\gamma'$, it is easily confirmed that $\tilde{c}_i=\frac{c_i}{k}$.
Next, we have $k=g(\dot\gamma,\dot\gamma)=\dot{x}^2\,g(\gamma',\gamma')$ and thus, as we work in dimension~2,
\[
	\dot{x}^2 = \frac{k}{H(x,y,\gamma'_1=1,\gamma'_2=y_x)}\,.
\]
The right hand side of this expression depends only on known data.
	
Let us consider the simplest possible example, i.e.\ the metric $g=(x^2+y)dxdy$. We find $H=g(\dot{\gamma},\dot{\gamma})=(y+x^2)\dot{x}\dot{y}$. Using $\tilde{I}_1=\tilde{c}_1$ and~\eqref{eqn:yx.quotient.solution}, we thus obtain
\begin{equation}\label{eqn:x.dot}
	\dot{x}^2
	= \frac{k}{(x^2+y)\,y_x}
	= \frac{2x-\tilde{c}_1}{(x^2+y)^2}\,k
\end{equation}
Finally, again employing~\eqref{eqn:yx.quotient.solution},
\begin{equation}\label{eqn:y.dot}
	\dot{y}^2
	= \dot{x}^2(y_x)^2
	= k\,\frac{2x-\tilde{c}_1}{(x^2+y)^2}\,\frac{(x^2+y)^2}{(2x-\tilde{c}_1)^2}
	= \frac{k}{(2x-\tilde{c}_1)}
\end{equation}
Equations~\eqref{eqn:x.dot} and~\eqref{eqn:y.dot} constitute a system of ODEs that determine the parametrization $x(t),y(t)$, given only the Hamiltonian as additional data. However, in order to actually solve for $x(t),y(t)$, Equation~\eqref{eqn:y.dot} is not needed. Indeed, having solved~\eqref{eqn:algebraic.quotient.solution} algebraically, we may substitute this solution into~\eqref{eqn:x.dot}. The resulting ODE for $x=x(t)$ has a unique solution. Subsequently, $y=y(t)$ is determined by~\eqref{eqn:algebraic.quotient.solution}, i.e.\ the algebraic equation~\eqref{eqn:algebraic.quotient.solution} is complemented only by the ODE~\eqref{eqn:x.dot}, yielding together the full solution $x(t),y(t)$.

\section{Proof of Theorem~\ref{thm:distinguished.coordinates}}\label{sec:proof.distinguished.coordinates}
Each point in the metrizability space~$\metrSp$ gives rise to a metric and then an integrable system (using the Benenti tensor~\eqref{eqn:benenti.tensor}), see Section~\ref{sec:connections.and.benenti.tensors}.
It is therefore interesting to discuss the geometry of this space in the presence of a projective symmetry $X$. We specifically look at the case when $X$ is essential.

\subsection{Part (a)}\label{sec:proof.distinguished.coordinates.a}

Let $\dim(\solSp)=2$. We prove that the number of eigenvalues of $\lie_X$ determines the number of connected components of $\essSp$. More precisely:
The space of solutions to the metrization equations~\eqref{eqn:linear.system.2D}, restricted to solutions for which the projective symmetry is actually essential, is isomorphic to the complement of an algebraic variety in~$\mathds{R}^\mathcal{D}$ where~$\mathcal{D}$ is the degree of mobility.
The proof is indeed easy:
\begin{proof}[Proof of part~(a) of Theorem~\ref{thm:distinguished.coordinates}]
 Each real eigenvalue $\Lambda_i$ gives rise to a 1-dimensional eigenspace $E_i\subseteq\solSp$ of~$\lie_X$ in the Liouville metrization space. These eigenspaces have, by definition, vanishing intersection with the essential metrizability space, $(\bigcup_iE_i)\cap\essSp=\emptyset$.
 Any $\sigma\in\metrSp$ corresponds to a metric, and thus we can conclude $\bigcup_iE_i=\solSp\setminus\essSp$.
 Recall that we have degree of mobility~2, that each $E_i$ is 1-dimensional, and that any two eigenspaces are orthogonal (in the obvious sense). Therefore, each~$E_i$ divides~$\solSp$ into two connected components, and so forth, yielding the required statement.
 Note, in particular, that if there do not exist real eigenvalues, $\bigcup_iE_i=0$ and thus $\essSp=\metrSp=\solSp\setminus\{0\}\sim\mathbb{R}^2\setminus\{0\}$, resulting in $\essSp\sim\mathbb{S}^1\times\mathbb{R}$.
 The interpretation of the components in this product, however, differs according to the parameter~$\lambda$, see Table~\ref{tab:mobility2} and part~(b) below.
\end{proof}

\noindent Note that the metrization space is open both w.r.t.\ the Zariski topology and the usual topology inherited from~$\mathds{R}^{\mathcal{D}}$.

\subsection{Part (b)}\label{sec:proof.distinguished.coordinates.b}

We have already seen in Proposition~\ref{prop:dini} that there exist natural (Dini) coordinates $(u_1,u_2)$ on the Liouville metrization space.
%
The goal of the current section is to derive distinguished coordinates on each connected component of the essential metrizability space $\essSp\subseteq\metrSp$, in the case of degree of mobility 2. To this end we employ orbital invariants of the projective flow.\footnote{The following discussion takes up and extends material included in versions~1 and~2 of the preprint arXiv:1705.06630 of~\cite{manno_2018}, which fell victim to subsequent shortenings of this paper. In particular, Lemmas~\ref{la:distinguished.coords.1}---\ref{la:distinguished.coords.3} and their proofs are found already in version~1 of arXiv:1705.06630.}
More precisely, we have already seen that $\lie_X$ acts on $\solSp$ so we look for coordinates such that one coordinate parametrizes the orbit while the other coordinate identifies a specific orbit.
First, recall that~$\lie_X$ is a linear action on the space~$\solSp$, see Section~\ref{sec:group.actions.symmetries}.
Its Gram matrix assumes, in a suitable basis, a Jordan normal form. For a two-dimensional matrix, there are three possible real Jordan normal forms: A diagonal matrix (2 real eigenvalues), a matrix with two complex conjugate eigenvalues, and an (upper triangular) Jordan block structure. In the latter case, we can further assume that the eigenvalue is 1, see Lemma~\ref{la:normal.forms.Lw} for details.

\subsubsection{Diagonal matrix case.}\label{sec:proof.distinguished.coordinates.b.1}
In this case, the Lie derivative~$\lie_X$ takes the form of a diagonal matrix, like~(I) of~\eqref{eqn:normal.forms.Lw}.
Let us prove the following statement, which implies the required part of Theorem~\ref{thm:distinguished.coordinates}:

\begin{lemma}\label{la:distinguished.coords.1}
The pairs $(u_1,u_2)$ and $(u_1',u_2')$ are related by the projective flow if and only if they satisfy one of the following cases
\[
	  (a)\quad \frac{u_1}{u_1'}>0\quad\text{and}\quad u_2=u_2'=0
	  \qquad\qquad\qquad\qquad
	  (b)\quad \frac{u_2}{u_2'}>0\quad\text{and}\quad \frac{u_1}{|u_2|^\lambda}=\frac{u_1'}{|u_2'|^\lambda}
	  =:u\,.
\]
\end{lemma}
\begin{proof}
In order to prove this lemma, let us assume $u_1', u_1, u_2', u_2\ne0$ (otherwise the statement is trivial). Now, let us assume that $u_1'=e^{\lambda t}u_1$ and $u_2'=e^t u_2$ holds, that in particular implies $\frac{u_2}{u_2'}>0$. Exponentiating the absolute values in the second equation by $\lambda\ne0$ and then eliminating $t$ we obtain
\begin{equation*}\label{eqn:proof.projcurve_II.2}
	  u_1'\,|u_2|^\lambda = u_1\,|u_2'|^\lambda\,,
\end{equation*}
i.e. the case $(b)$ above.
Conversely, if we assume either $(a)$ or $(b)$, we arrive, by a straightforward computation, to the required condition.
\end{proof}

\noindent The graphs in Table~\ref{tab:mobility2} illustrate the orbits of the projective flow inside $\essSp\subset\solSp$.
In conclusion, we arrive at the following: Along a given orbit of the projective flow, the function~\smash{$u = \frac{u_1}{|u_2|^\lambda}$} is constant. Thus we can use this number (up to its sign) as a coordinate on the respective connected component of $\essSp\subset\metrSp$. Note, as well, that the signs of $u_1$ and $u_2$ identify the quadrant of $\mathbb{R}^2\sim\solSp$ and thus indicate the respective connected component. Altogether, we obtain the transformation
\[
  u_1 = u\,e^{\lambda s}\,,\quad
  u_2 = e^s\,.
\]

\begin{remark}\label{rmk:invariants.diagonal.case}
	Note that the same strategy works in any degree of mobility.
	For instance, let $(\sigma_1,\sigma_2,\dots,\sigma_m)$ be a basis of $\solSp$
	such that $\lie_X\sigma_i = \lambda_i\sigma_i$.
	Let us denote the flow of~$X$ by $\phi_t$. We have
	\begin{equation*}
	  \phi_t^*\left(\sum u_{i}\sigma_i\right)
	  = \sum u_{i}e^{\lambda_i t}\,\sigma_i =: \sum u_{i}'\sigma_i\,,
	\end{equation*}
	and thus $u_{i}'=u_{i}e^{\lambda_i t}$. Requiring $\lambda_i\ne0$ (existence of a non-Killing basis) and $u_{i}\ne0$, it is then quickly confirmed that
	\begin{equation*}
	\left|\frac{u_{i}'}{u_{i}}\right|^{\lambda_j} = \left|\frac{u_{j}'}{u_{j}}\right|^{\lambda_i}\,,
	\quad\text{and therefore we have $\frac{m(m-1)}{2}$ functions}\quad
	F_{ij} = \frac{\left|u_{i}\right|^{\lambda_j}}{\left|u_j\right|^{\lambda_i}}\quad
	\text{for $i<j$}\,,
	\end{equation*}
	which are constant along orbits. However, not all of them are independent.
\end{remark}

\paragraph{Alternative (angular) invariants.}
We could also use more intuitive functions that are constant along orbits of the projective symmetry.
Indeed, take the diagonal matrix case and consider the pullback
\[
\phi_t^*\left(\sum u_{i}\sigma_i\right)
= \sum e^{\lambda_it}u_{i}\sigma_i
= \sum u_{i}'\sigma_i
\]
After a change to ellipsoidal coordinates,
$(u_1,u_2) = (\,e^{-\lambda_1r}\,\cos(\varphi), e^{-\lambda_2r}\,\sin(\varphi)\,)$,
\begin{align*}
\phi_t^*\left(\sum u_{i}\sigma_i\right)
&= e^{\lambda_1(t-r)}\cos(\varphi)\,\sigma_1 + e^{\lambda_2(t-r)}\sin(\varphi)\,\sigma_2 \\
&= e^{\lambda_1(t')}\cos(\varphi)\,\sigma_1 + e^{\lambda_2(t')}\sin(\varphi)\,\sigma_2
\stackrel{t=r}{=} \left( \cos(\varphi)\,\sigma_1 +\sin(\varphi)\,\sigma_2 \right)
\end{align*}
and thus we have established $\varphi$ as an orbital invariant, while $t\to t'=t-r$.
We will come back to this idea in Section~\ref{sec:proof.superintegrable}.

\subsubsection{Upper triangular Jordan block case.}\label{sec:proof.distinguished.coordinates.b.2}
Let us now assume that~$\lie_X$ takes the form~(II) of~\eqref{eqn:normal.forms.Lw}.
Exponentiating, we find
\begin{equation}
 \sum u_{i}'\sigma_i
 =\phi_t^*\left(\sum u_{i}\sigma_i\right)
 = \begin{pmatrix}
    e^t & te^t \\
        & e^t
  \end{pmatrix}\,\begin{pmatrix} \sigma_1\\ \sigma_2\end{pmatrix}
\end{equation}
which implies, in particular, $u_2'=e^tu_2$ and $u_1'=e^tu_1+te^tu_2$, and thus
\[
 u_1' = \frac{u_2'}{u_2}\,u_1+\ln\left(\frac{u_2'}{u_2}\right)\,u_1'
\]
As a consequence,
\begin{equation*}
 \frac{e^{\nicefrac{u_1'}{u_2'}}}{u_2'}
 =\frac{e^{\nicefrac{u_1}{u_2}}}{u_2}\,.
\end{equation*}
This implies that the function~$F_1(u_1,u_2)=\frac{e^{\nicefrac{u_1}{u_2}}}{u_2}$ is constant along orbits of the projective symmetry.

\begin{lemma}\label{la:distinguished.coords.2}
	The pairs $(u_1,u_2)$ and $(u_1',u_2')$ are related by the projective flow if and only if they satisfy one of the following cases
	\[
	   (a)\quad u_1'=u_1=0 \quad\text{and}\quad \frac{u_2'}{u_2}>0
	   \qquad\qquad\qquad
		(b)\quad \frac{u_1'}{u_1}>0 \quad\text{and}\quad
		 \frac{1}{u_1'}e^\frac{u_2'}{u_1'}=\frac{1}{u_1}e^\frac{u_2}{u_1}
		 =:u\,.
	\]	
\end{lemma}
\begin{proof}
	We have $u_1'=u_1\,e^t$ and $u_2'=u_1\,te^t+u_2\,e^t$. Eliminating $t$ we therefore have either $u_1'=u_1=0$ and then $\frac{u_2'}{u_2}>0$, or $\frac{u_1'}{u_1}>0$ and
	\[
	  u_2'
	  =u_1\,\frac{u_1'}{u_1} \ln\left(\frac{u_1'}{u_1}\right)
	    +u_2\,\frac{u_1'}{u_1}\,.
	\]
	The converse direction is straightforward.
\end{proof}

\noindent We have thus obtained, letting $u_1=e^s$, that $u_2=u_1\,\ln(u_1u)=e^s\ln(e^su)=e^s\,(s+\ln(u))$.
The plots in Table~\ref{tab:mobility2} illustrates the orbits of a metric in the projective class $\projcl(\g1)=\projcl(\g2)$ under the flow of $\phi_t$.

\begin{remark}
The above reasoning can in fact be extended to higher degree of mobility. Assume, for instance,
\[
  \lie_X\sigma_m = \sigma_m\,,
  \qquad\text{and}\qquad
  \lie_X\sigma_i=\sigma_i+\sigma_{i+1}
  \quad\text{for $1\leq i\leq m-1$}\,.
\]
As a consequence, we obtain
\begin{equation*}
 \sum u_{i}'\sigma_i
 =\phi_t^*(\sum u_{i}\sigma_i)
 = \begin{pmatrix}
    e^t & te^t & \cdots & \frac{e^tt^{m-1}}{(m-1)!} \\
      & e^t & \dots & \vdots \\
      & & \ddots & te^t \\
      & & & e^t
  \end{pmatrix}\,\begin{pmatrix} \sigma_1\\ \sigma_2\\ \vdots\\ \sigma_{m}\end{pmatrix}\,,
\end{equation*}
implying in particular that
$
 u_m'=e^tu_m
 \quad\text{and}\quad
 u_{m-1}'=e^tu_{m-1}+te^tu_m
$.
Thus, assuming $u_2,\dots,u_m\ne0$ (non-homotheticness),
\[
 u_{m-1}' = \frac{u_m'}{u_m}\,u_{m-1}+\ln\left(\frac{u_m'}{u_m}\right)\,u_1'
\]
from which we infer
$\frac{e^{\nicefrac{u_{m-1}'}{u_m'}}}{u_m'}=\frac{e^{\nicefrac{u_{m-1}}{u_m}}}{u_m}$,
i.e.\ the function~$\frac{e^{u_{m-1}/u_m}}{u_m}$ is constant along orbits of the projective symmetry.
Similarly, the other equations yield further invariants, e.g., using
$t=\frac{u_{m-1}'}{u_m'}-\frac{u_{m-1}}{u_m}$
we derive from
\[
 \frac{u_{m-2}'}{u_m'}
 = \frac{u_{m-2}}{u_m} + t\,\frac{u_{m-1}}{u_m} + \frac12\,t^2
 = \frac{u_{m-2}}{u_m} + t\,\left(\frac{u_{m-1}}{u_m}+\frac{t}{2}\right)
\]
the function~$\frac{u_{m-2}}{u_m}-\frac{u_{m-1}^2}{u_m^2}$, which is constant along orbits of the projective symmetry, and so forth.
\end{remark}

\paragraph{Other invariants: higher multiplicity eigenvalues.}
We have, recalling case~(II) of Section~\ref{sec:group.actions.symmetries},
\[
\phi_t^* \left( \sum_1^2 u_{i}\sigma_i \right)
= u_1 e^t \sigma_1 + u_1 t e^t \sigma_2 + u_2 e^t \sigma_2
= u_1' \sigma_1 + u_2' \sigma_2\,,
\]
which implies new coordinates $(r,\kappa)$ with
$(u_1,u_2)=(e^{r},re^r\,\kappa)$,
i.e.\ $\kappa:=\frac{u_2}{u_1}-\ln(u_1)=\ln(F_1)$ where $F_1=F_1(u_1,u_2)$ is as defined in Section~\ref{sec:proof.distinguished.coordinates.b.2}.
This means
\[
\phi_t^* \left( \sum_1^2 u_{i}\sigma_i \right)
= e^{t+r}\,\sigma_1 + e^{t+r}\,(t+r+\kappa)\,\sigma_2\,,
\]
i.e.\ $\kappa$ is the invariant parameter and $r\to r'=r+t$ parametrizes the projective orbits.

\subsubsection{Pair of complex eigenvalues.}\label{sec:proof.distinguished.coordinates.b.3}
This is the case when $\lie_X$ has the form~(III), i.e.
\begin{equation}\label{eq:III.expLw}
\exp(t\lie_X)=e^{\lambda t}\,\begin{pmatrix} \cos(t) & -\sin(t)\\ \sin(t) & \quad\cos(t) \end{pmatrix}.
\end{equation}
This case has already been discussed in~\cite{manno_2018}, but we restate it here for completeness.

Consider a general metric $g$, obtained from $u_1\sigma_1+u_2\sigma_2$ via Formula~\eqref{eqn:sigma.g}, realizing the projective class $\projcl(\g1)=\projcl(\g2)$.
By virtue of Equation~\eqref{eq:III.expLw}, we can compute
\begin{equation}\label{eq:III.pullback}
\phi_t^*(u_1\sigma_1+u_2\sigma_2)=e^{\lambda t}\ \left( u_1\cos(t)\sigma_1-u_1\sin(t)\sigma_2+u_2\sin(t)\sigma_1+u_2\cos(t)\sigma_2 \right).
\end{equation}
Thus, the orbits of the projective flow describe logarithmically spiralling curves in \solSp.
We obtain
\[
 u_1' = \left( u_1\cos(t)+u_2\sin(t) \right)\,e^{\lambda t}\,,\quad
 u_2' = \left( u_2\cos(t)-u_1\sin(t) \right)\,e^{\lambda t}\,,
\]
and thus
\begin{equation}\label{eqn:pre.invariant.2cpx}
 (u_1')^2+(u_2')^2
 = e^{2\lambda\,t}\,\left( u_1^2+u_2^2 \right)\,.
\end{equation}
Therefore, if~$\lambda=0$, $u_1^2+u_2^2$ is constant along orbits of the projective symmetry.
If, in contrast, $\lambda\ne0$, consider
\[
 \frac{u_1'}{u_2'} = \frac{u_1+u_2\tan(t)}{u_2-u_1\tan(t)}\,,
 \quad\text{which becomes}\quad
 \tan(t)=\tan(s'-s)\,,
\]
where we introduce $s = \arctan\left(\frac{u_1}{u_2}\right)$.
Thus we find $t = s'-s+N\pi$, $N\in\mathbb{Z}$,
and therefore Equation~\eqref{eqn:pre.invariant.2cpx} becomes
$((u_1')^2+(u_2')^2)\,e^{-2\lambda s'}
 = e^{-2\lambda s}\,\left( u_1^2+u_2^2 \right)\ e^{2N\lambda\pi}$.
We conclude that the function
\begin{equation*}
 \left[ \frac{\ln(u_1^2+u_2^2)}{2\lambda}-s \right]\!\!\mod\,\pi
\end{equation*}
is constant along the orbits of the projective symmetry.
Note that this implies new, spiralling parameters $(s,u)$ in $\solSp$, namely
\begin{equation}\label{eqn:CIII.reparametrization}
 u_1=u e^{\lambda s} \sin(s)
 \qquad\text{and}\qquad
 u_2=u e^{\lambda s} \cos(s)
\end{equation}
with the inverse reparametrization given by
$s=\arctan\left(\frac{u_1}{u_2}\right)$
and
$u=\sqrt{u_1^2+u_2^2}\,e^{-\lambda\arctan\left(\nicefrac{u_1}{u_2}\right)}$.

\begin{remark}
Note that in case $\lambda=0$, the new parametrization is by polar parameters, with $u>0$ being the radius parameter and $s\in\angl$ being the angle parameter.
On the other hand, if $\lambda\ne0$, it is convenient to take $s\in\mathbb{R}$ as parameter along the spiralling curves in $\solSp$, whereas $u$ should be considered as an angle by way of letting $u=e^{\lambda\alpha}$ where $\alpha\in\angl$.
\end{remark}
\noindent The parametrization~\eqref{eqn:CIII.reparametrization} is adjusted to the problem in the sense that $K$ is constant along orbits of the projective flow.

\begin{lemma}\label{la:distinguished.coords.3}
The pairs $(s,u)$ and $(s',u')$ are related by the transformation~\eqref{eq:III.pullback} if and only if they satisfy $u' = u$.
\end{lemma}
\begin{proof}
	Consider~\eqref{eq:III.pullback}, i.e.
	\begin{multline*}
	  u e^{\lambda (t+s)} \left(
	    (\sin(s)\cos(t)+\cos(s)\sin(t))\,\sigma_1
	    +(-\sin(s)\sin(t)+\cos(s)\cos(t))\,\sigma_2
	  \right) \\
	  =
	  u' e^{\lambda s'} \sin(s')\sigma_1
	  +u' e^{\lambda s'} \cos(s')\sigma_2 \,.
	\end{multline*}
	Using standard trigonometrical identities, one finds $u'=u$ and $s'=s+t$
	(in case $\lambda=0$ we understand the second equality modulo $2\pi$).
\end{proof}

\noindent The results of Sections~\ref{sec:proof.distinguished.coordinates.b.1}---\ref{sec:proof.distinguished.coordinates.b.3}, particularly Lemmas~\ref{la:distinguished.coords.1}---\ref{la:distinguished.coords.3}, prove part~(b) of Theorem~\ref{thm:distinguished.coordinates}.
One can make use of these distinguished coordinates to find normal forms for metrics with one essential normal forms. This procedure has indeed already been executed, for the case of Section~\ref{sec:proof.distinguished.coordinates.b.3}, in~\cite{manno_2018}.

\section{Proof of Theorem~\ref{thm:superintegrable.systems}}\label{sec:proof.superintegrable}
Let $g$ be a metric with one, essential projective symmetry and let its degree of mobility be~3. Let $g$ correspond to the solution $\sigma\in\solSp$ via~\eqref{eqn:sigma.g}.
Parametrize the space~$\solSp$ of solutions to~\eqref{eqn:linear.system.2D} using the generators~\eqref{eqn:supint.generators} and~\eqref{eqn:sigma.g}.
Now consider the pullback of $\sigma=\sum u_i\sigma_i$ along the flow~$\phi_t$ of the projective symmetry~$X$, as in~\eqref{eqn:pullbacks.dom2}, i.e.
\begin{equation}\label{eqn:proj.orbit.supint}
\phi_t^*\sigma
= \phi_t^*\left(\sum u_{i}\sigma_i\right)
= e^{-\frac53t}u_1\sigma_1+e^{-\frac23t}u_2\sigma_2+e^{\frac43t}u_3\sigma_3\,.
\end{equation}
Reparametrizing the essential metrizability space in terms of ellipsoidal coordinates,
\begin{equation}\label{eqn:parametrization.dom.3}
\begin{pmatrix} u_1 \\ u_2 \\ u_3 \end{pmatrix}
= \begin{pmatrix}
e^{-\frac53r}\,\sin(\theta)\,\cos(\varphi) \\
e^{-\frac23r}\,\sin(\theta)\,\sin(\varphi) \\
e^{\frac43r}\,\cos(\theta)
\end{pmatrix}\,,
\end{equation}
where $\theta\in(0,\pi)$ and $\varphi\in(0,2\pi)$ with $\varphi\not\in\{\tfrac12\pi,\pi,\tfrac32\pi\}$,
we achieve representatives on the unit sphere, by letting $t=-r$.
Any two such representatives are mutually non-isometric, i.e.\ not linkable by a local coordinate transformation~\cite{manno_2018}.

\begin{remark}[Orbital invariants]
In Remark~\ref{rmk:invariants.diagonal.case}, we have found orbital invariants \smash{$F_{ij}=\frac{|u_{i}|^{\lambda_j}}{|u_j|^{\lambda_i}}$}, $1\leq i<j\leq 3$. These invariants exist, analogously, for the case considered here. However, we actually do not use these invariants in our reasoning of this current section. Indeed, such a choice of invariants is less geometrically instructive as the one actually pursued, which the reader might easily verify by computation.
Having said this, it should also be stated that the parametrization~\eqref{eqn:parametrization.dom.3} indeed reveals better invariants, namely~$\theta$ and~$\varphi$. These invariants have been discussed at the end of Section~\ref{sec:proof.distinguished.coordinates.b.1} and are better suited (for the current purposes) than the $F_{ij}$ in at least two respects: Firstly, they are geometrically much more instructive as they can be identified with angles on the 2-sphere. Secondly, they are ``economical'' and indeed it is quickly verified that the invariants $F_{ij}$ cannot all be independent.
\end{remark}

\noindent So far, we have achieved that the superintegrable systems that admit exactly one, essential projective vector field, are parametrized by points on the unit sphere.\smallskip

\noindent An explicit description of these superintegrable systems is obtained as follows:
\begin{proof}[Proof of Theorem~\ref{thm:superintegrable.systems}]
It is easily verified that $\sigma$ and $\bar{\sigma}$ are orthogonal in $\solSp$, using the standard scalar product of $\solSp\sim\R^3$ w.r.t.\ the basis $(\sigma_1,\sigma_2,\sigma_3)$.
Thus, the third Liouville tensor is obtained from their cross product, $\hat{\sigma}=\sigma\times\bar{\sigma}$. This defines three linearly independent integrals of motion, the first of which is the Hamiltonian of~$g=g[\theta,\varphi]$.
These integrals $(H,I,J)$ are given via~\eqref{eqn:integral.g},
\[
  H = \frac12\,g^{ij}p_ip_j\,,\quad
  I = \left(\frac{\det(g)}{\det(\bar{g})}\right)^{\frac23}\,\bar{g}^{ij}p_ip_j\,,\quad
  J = \left(\frac{\det(g)}{\det(\hat{g})}\right)^{\frac23}\,\hat{g}^{ij}p_ip_j\,,
\]
where $g,\bar{g},\hat{g}$ correspond to $\sigma,\bar{\sigma},\hat{\sigma}$, respectively via Formula~\eqref{eqn:sigma.g}. Raising and lowering indices is done by using the metric $g$.
We have to verify that these integrals are functionally independent, i.e.\ that the matrix
\begin{equation}\label{eqn:jacobian.matrix}
  A[g,\bar{g},\hat{g}] =
      \begin{psmallmatrix}
       H_x & I_x & J_x \\
       H_y & I_y & J_y \\
       H_{p_x} & I_{p_x} & J_{p_x} \\
       H_{p_y} & I_{p_y} & J_{p_y}
      \end{psmallmatrix}
\end{equation}
has maximal rank, i.e.\ rank~3.
In fact, it is enough to check this for the integrals $(H,I,J)$ obtained from $(\g1,\g2,\g3)$ of~\eqref{eqn:supint.generators}, because the isomorphism among the superintegrable systems does not affect the momenta derivatives, c.f.\ \cite{bryant_2008}. From this latter fact, one can deduce that, if \eqref{eqn:jacobian.matrix} had rank less than~3, for $g\ne\g1$, the following equation would necessarily have to hold:
\[
  (\gamma_x-\gamma_y)\,\det(A[\g1,\g2,\g3]) = 0\,,
\]
where $\gamma=\frac{\det(g)}{\det(\g1)}$.
It is straightforward to show that $\gamma_x\ne\gamma_y$ for $g\ne\g1$.
Thus, it remains to prove $\det(A[\g1,\g2,\g3])\ne0$, generically. We take $(H,I,J)$ for $(g,\bar{g},\hat{g})=(\g1,\g2,\g3)$ from \eqref{eqn:supint.generators}, and we consider the $3\times3$ submatrix $S$ of $A[\g1,\g2,\g3]$,
\[
  S =
  \begin{psmallmatrix}
  H_y & I_y & J_y \\
  H_{p_x} & I_{p_x} & J_{p_x} \\
  H_{p_y} & I_{p_y} & J_{p_y}
  \end{psmallmatrix}
\]
of \eqref{eqn:jacobian.matrix}.
Assume $\rk(A)<3$, then $\det(S)=0$.
But $\det(S)$ is, after multiplying through with the common denominator, a polynomial in $x,y,p_x,p_y$. It is easily confirmed that it cannot be zero for generic values of these variables.
\end{proof}

\begin{remark}
What about the remaining six points, i.e.\ points on the 2-sphere that correspond, via~\eqref{eqn:sigma.g}, to metrics with a homothetic (as opposed to essential) projective vector field?
Indeed, Equations~\eqref{eqn:new.superintegrable.basis} apply also for these metrics, as the following table demonstrates.
\begin{center}
\begin{tabular}{l|lll}
 \toprule
 \textbf{Point} & \multicolumn{3}{l}{\textbf{Basis $(\sigma,\bar\sigma,\hat\sigma)$}} \\
 \midrule
 $\theta=\frac{\pi}{2}, \varphi=0$ &
 $\sigma[\nicefrac{\pi}{2},0] = \sigma_1$, &
 $\bar{\sigma}[\nicefrac{\pi}{2},0] = \sigma_3$, &
 $\hat{\sigma}[\nicefrac{\pi}{2},0] = \sigma_2$
 \\
 $\theta=\frac{\pi}{2}, \varphi=\frac{\pi}{2}$ &
 $\sigma[\nicefrac{\pi}{2},\nicefrac{\pi}{2}] = \sigma_2$, &
 $\bar{\sigma}[\nicefrac{\pi}{2},\nicefrac{\pi}{2}] = -\sigma_3$, &
 $\hat{\sigma}[\nicefrac{\pi}{2},\nicefrac{\pi}{2}] = -\sigma_1$
 \\
 $\theta=0$ &
 $\sigma[0,\varphi] = \sigma_3$, &
 $\bar{\sigma}[0,\varphi]=\cos(\varphi)\,\sigma_1+\sin(\varphi)\,\sigma_2$, &
 $\hat{\sigma}[0,\varphi] = -\sin(\varphi)\,\sigma_1+\cos(\varphi)\,\sigma_2$ \\
 \bottomrule
\end{tabular}
\end{center}
In the last case, we are free to choose the parameter $\varphi$, say $\varphi=0$ for convenience. We thus obtain the simpler form $\bar{\sigma}[0,0] = \sigma_1$, $\hat{\sigma}[0,0] = \sigma_2$.
The other three exceptional points can be checked analogously and we therefore confirmed that the projective class of metrics III(C) of~\cite{manno_2018} is indeed parametrized by the unit sphere, including these six points where the essential projective symmetry degenerates into a homothetic one.
\end{remark}

\noindent We have thus confirmed: Metrics with one, essential projective vector field and degree of mobility~3 are superintegrable in the usual sense and their projective class is parametrized by the 2-sphere. All points on this sphere correspond to such superintegrable systems, except for six points. The six exceptional points are given by the intersection of the sphere with the axes in Dini coordinates, i.e.\ eigenspaces of the Lie derivative w.r.t.\ the projective vector field. They correspond to metrics for which the projective symmetry degenerates into a homothetic one.

The present paper considers free Hamiltonian systems given by the (super-)integrable metrics with one, essential projective vector field. A subsequent paper along these lines is~\cite{vollmer_2018}, which investigates the potentials admitted by the superintegrable systems, i.e.\ when the Hamiltonian is $H=\frac12\,g^{ij}p_ip_j+V$, admitting a potential function~$V:T^*M\to\mathbb{R}$.

\appendix
\section{Metrics with one, essential projective symmetry: Normal forms for degree of mobility 2}\label{app:normal.forms.dom.2}

Let $g$ be a (pseudo-)Riemannian metric on a $2$-dimensional manifold $M$ of degree of mobility~2. Let us assume that $\dim(\projalg(g))=1$ in any neighborhood of $M$ and that $\projalg(g)$ is generated by an essential projective vector field. Then, in a neighborhood of almost every point there exists a local coordinate system $(x,y)$ such that~$g$ assumes one of the mutually non-diffeomorphic normal forms given in Table~\ref{tab:normal.forms.dom.2} (organized according to their type, see Proposition~\ref{prop:dini}).
Table~\ref{tab:normal.forms.dom.2} contains some ``special cases'' for the choice of the parameter values. For the origin of these exceptions, see~\cite{manno_2018}. Roughly speaking, they reveal additional symmetries of the geometries, which manifest themselves as a freedom to ``swap'' the variables on the base manifold.

\afterpage{%
\begin{table}[!ht]
\begin{center}
\textbf{Normal forms for metrics with one, essential projective vector field}
\smallskip

\begin{tabular}{cg{2.9cm}v{8.5cm}g{2.5cm}}
 \toprule
 \textbf{Label} & \textbf{Parameters} & \textbf{Normal form} & \textbf{Coordinates} \\
 \midrule
 (A.1)
 & $\varepsilon,h,\xi$
 & \begin{minipage}{8.5cm}
    \vskip 0.6em
    $g=\kappa\left( \frac{(e^{\xi x}-he^{\xi y})\,e^{2x}}{(1+\varrho he^{\xi y})(1+\varrho e^{\xi x})^2}\,dx^2 +\varepsilon\frac{(e^{\xi x}-he^{\xi y})\,e^{2y}}{(1+\varrho he^{\xi y})^2(1+\varrho e^{\xi x})}\,dy^2 \right)$
   \end{minipage}
 & $\kappa\in\Rnz$, $\varrho\in\pmo$ \\
 \midrule
 (A.2)
 & $h$
 & $g=\kappa\left(\frac{(y-x)e^{-3x}}{x^2y}dx^2 + \frac{h (y-x)e^{-3y}}{xy^2}dy^2\right)$
 & $\kappa\in\Rnz$ \\
 \midrule
 (A.3a)
 & $|h|\leq e^{-3\lambda\pi}$,\hfill $\lambda>0$
 & $g=\frac{\sin(y-x)}{\sin(y+\theta)\,\sin(x+\theta)}\,
		\left(
		\frac{e^{-3\lambda\,x}}{\sin(x+\theta)}\,dx^2
		+\frac{h\,e^{-3\lambda\,y}}{\sin(y+\theta)}\,dy^2
		\right)$
 & $\theta\in\angl$ \\
 (A.3b)
 & $h$,\hfill $\lambda=0$
 & $g=\kappa\,\frac{\sin(y-x)}{\sin(y)\,\sin(x)}\,
		\left(
		\frac{dx^2}{\sin(x)}
		+\frac{h\,dy^2}{\sin(y)}
		\right)$
 & $\kappa>0$ \\
 \midrule
 (B.4)
 & $C,\xi$
 & $g=\kappa\,\left(
	 \frac{Cz^\xi-\cc{C}\cc{z}^\xi}{(1+\cc{C}\cc{z}^\xi)(1+Cz^\xi)^2}\,dz^2
	 -\frac{Cz^\xi-\cc{C}\cc{z}^\xi}{(1+\cc{C}\cc{z}^\xi)^2(1+Cz^\xi)}\,d\cc{z}^2
	\right)$
 & $\kappa\in\Rnz$ \\
 \midrule
 (B.5)
 & $C$
 & $g=\kappa\,(\cc{z}-z)\,
				\left(
				\frac{C\,e^{-3z}}{z^2\cc{z}}\,dz^2
				 -\frac{\cc{C}\,e^{-3\cc{z}}}{z\cc{z}^2}\,d\cc{z}^2
				\right)$
 & $\kappa>0$ \\
 \midrule
 (B.6a)
 & $C$,\hfill $\lambda>0$
 & $g=\frac{\sin(\cc{z}-z)}{\sin(\cc{z}+\theta)\,\sin(z+\theta)}\,
			\left(
			\frac{C\,e^{-3\lambda z}}{\sin(z+\theta)}\,dz^2
			-\frac{\cc{C}\,e^{-3\lambda\cc{z}}}{\sin(\cc{z}+\theta)}\,d\cc{z}^2
			\right)$
 & $\theta\in\angl$ \\
 (B.6b)
 & $C\ne1$,\hfill $\lambda=0$
 & $g=\kappa\,\frac{\sin(\cc{z}-z)}{\sin(\cc{z})\,\sin(z)}\,
			\left(
			\frac{C\,dz^2}{\sin(z)}
			-\frac{\cc{C}\,d\cc{z}^2}{\sin(\cc{z})}
			\right)$
 & $\kappa>0$ \\
 \midrule
 (C.7)
 & $\xi\ne\frac12$
 & \begin{minipage}{8.5cm}
    \vskip 0.6em
    $g=\kappa\left(-2\frac{y^{\nicefrac{1}{\xi}}+x}{(y-\varrho)^3}dxdy
			+ \frac{1}{(y-\varrho)^4}\,\left(y^{\nicefrac{1}{\xi}}+x\right)^2\,dy^2\right)$
   \end{minipage}
 & $\kappa\in\Rnz$, $\varrho\in\pmo$ \\
 \midrule
 (C.8)
 & ---
 & $g=\kappa\,(Y(y)+x)\,dxdy$
   \hfill
   where $Y(y)=\int^y \frac{e^{\nicefrac{3}{2s}}}{|s|^{\nicefrac32}}\,ds$
 & $\kappa\in\Rnz$ \\
 \midrule
 (C.9a)
 & $\lambda>0$
 & $g = e^{\lambda\theta}\,(Y_\lambda(y)+x)\,dxdy$
   \hfill
   $Y_\lambda(y)= \int^y \frac{e^{-\nicefrac{3\lambda}{2}\,\arctan(s)}}{\sqrt[4]{s^2+1}}\,ds$
 & $\theta\in\angl$ \\
 (C.9b)
 & $\lambda=0$
 & $g = \kappa\,(Y_0(y)+x)\,dxdy$
 & $\kappa>0$ \\
 \bottomrule
\end{tabular}
\begin{flushleft}
          Universally, the constraints $\xi\in(0,1)\cup(1,4]$, $0\ne h\leq1$, $\varepsilon\in\pmo$ and $C=e^{i\varphi}$, where $\varphi\in\halfangl$, apply.
          Additional constraints are specified in the table. Furthermore, we impose some parameter constraints in individual cases.
\end{flushleft}
\begin{enumerate}[noitemsep,topsep=0pt,align=right,leftmargin=2cm,font=\colorbox{black!10}]
 \item[(A.1)] If $\xi=2$: $h\ne-\varepsilon$ and $h\ne-4\varepsilon$; if $\xi=3$: $|h|\ne1$ when $\varepsilon=-1$; if $\xi=4$: $h\ne1$. \\
	\emph{Special cases:} If $h=-1$: $\varrho=1$. If $h=+1$ and $\varepsilon=1$: $\kappa>0$.
 \item[(A.2)] \emph{Special case:} In case $h=1$, $\kappa>0$ is required.
 \item[(A.3)] If $\lambda=0$: $h\ne\pm1$.
	\emph{Special case:} If $\lambda>0$ and $|h|=e^{-3\lambda\pi}$, we require $\theta\in\halfangl$.
 \item[(B.4)] If $\xi=2$: $C\ne\pm1$; if $\xi=3$: $C^2\ne\pm1$; if $\xi=4$: $C\ne1$.
\end{enumerate}
\smallskip

\hrule

\caption{The normal forms for metrics with one, essential projective vector field. As established in \cite{matveev_2012,manno_2018}.}\label{tab:normal.forms.dom.2}
\end{center}
\end{table}
}

Let us briefly comment on the parameters $\lambda,\xi$ and $C,h,\varepsilon$ in Table~\ref{tab:normal.forms.dom.2}. In fact, we can split them up into three sets, which appear to have distinct effects on the respective normal forms, characterizing different aspects of the geometry.
\begin{table}[!ht]
\begin{center}
\begin{tabular}{ll}
 \toprule
 \textbf{Parameters} & \textbf{Interpretation} \\
 \midrule
 $\lambda$ and $\xi$ & Captures the eigenvalue configuration of $\lie_X$ \\
 $|h|$ and $C$ & `Measures' how far the metric is from the symmetry to swap variables $x,y$ \\
 $\varepsilon$ and $\sgn(h)$ & Contain information about the signature \\
 \bottomrule
\end{tabular}
\caption{The interpretation of the parameters of projective classes. Note that $\xi=2\,\frac{\lambda-1}{2\lambda+1}$ as derived in~\cite{matveev_2012}.}
\end{center}
\end{table}

\noindent Indeed, consider~\eqref{eqn:normal.forms.Lw} for the first pair of parameters. The other two can be inferred by directly considering the normal forms in Table~\ref{tab:normal.forms.dom.2}.
We also note that when these parameters ``hit'' a case of additional symmetry (e.g., $\lambda\to0$ or $|h|\to0$), the properties of the projective class may be affected. For instance, if $\lambda\to0$, we see in Figure~\ref{tab:mobility2} that the spiralling orbits of the projective symmetry ``degenerate'' to circular ones, changing the geometry of the metrizability space. In the case when $|h|=1$, we obtain special cases with additional symmetry (linked to the interchangability of the variables $x,y$) in cases (A.1) and (A.2). A similar special case appears in case (A.3a); in case (A.3b), in contrast, $|h|=1$ is a case where the projective symmetry becomes trivial, i.e.\ a Killing-Noether symmetry.

\section{Metrics with one, essential projective symmetry: Normal forms for degree of mobility 3}\label{app:normal.forms.dom.3}
Let $g$ be a (pseudo-)Riemannian metric on a 2-dimensional manifold~$M$, but now with degree of mobility~3. We still assume $\dim(\projalg(g))=1$ in any neighborhood of $M$ and that $\projalg(g)$ is generated by an essential projective vector field. Then, in a neighborhood of almost every point there exists a local coordinate system $(x,y)$ such that $g$ assumes the normal form ($\kappa\in\Rnz$, $\varepsilon\in\pmo$ and $c\in\mathds{R}$), see~\cite{manno_2018},
\begin{subequations}\label{eqn:dom.3.explicit.normal.form}
\begin{align}
  g &= \kappa\,\left(-2\frac{y^2+x}{(y-\varepsilon)^3}\,dxdy+\frac{(y^2+x)^2}{(y-\varepsilon)^4}\,dy^2\right)
  \label{eqn:dom.3.nf.1} \\
  g &= \frac{\kappa}{F(0,\varepsilon;x,y)^2}\,\left(
		    9(y^2+x)^2\,dx^2
		    -2(y^3+9xy-2\varepsilon)(y^2+x)\,dxdy
		    +12x(y^2+x)^2\,dy^2
	\right)
  \label{eqn:dom.3.nf.2} \\
  g &= \frac{\kappa}{F(\varepsilon,c;x,y)^2}\,\left(
		 9(y^2+x)^2\,dx^2
		 -2(y^3+2\varepsilon\,y+9xy-2c)(y^2+x)\,dxdy
		 +4(\varepsilon+3x)(y^2+x)^2\,dy^2
		\right)
  \label{eqn:dom.3.nf.3}
\end{align}
with
\[
 F(\zeta,c;x,y) = y^6-9xy^4+27x^2y^2-27x^3+4c^2-(36xy+4y^3)\,c
		  +(18xy^2-5y^4-9x^2-8cy)\,\zeta+4y^2\,\zeta^2.
\]
\end{subequations}
Two such normal forms, for different parameter values, are not diffeomorphic.
There is, also, another approach, which may be considered more geometrical. As we have degree of mobility~3, the space of solutions~$\solSp$ to~\eqref{eqn:linear.system.2D} is spanned by three basis vectors, which via~\eqref{eqn:sigma.g} correspond to three metrics. These metrics can be taken as follows \cite{matveev_2012}:
\begin{subequations}\label{eqn:supint.generators}
	\begin{align}
	\g1 &= (y^2+x)\,dxdy\label{eqn:y2.plus.x} \\
	\g2 &= -2\,\tfrac{y^2+x}{y^3}\,dxdy+\tfrac{(y^2+x)^2}{y^4}\,dy^2\label{eqn:y2.plus.x.2} \\
	\g3 &= \tfrac{y^2+x}{(3x-y^2)^6}\,
	(9\,(y^2+x)\,dx^2-4y\,(9x+y^2)\,dxdy+12x\,(y^2+x)\,dy^2)\label{eqn:y2.plus.x.3}
	\end{align}
\end{subequations}
Indeed, this choice is (up to multiplication by a constant and reordering of the basis) canonical in the following sense: These metrics correspond, again via~\eqref{eqn:sigma.g} to eigenvectors of the Lie derivative $\lie_X$ w.r.t.\ the projective vector field~$X$. Thus, $X$ is homothetic for the metrics~$\g{i}$ in~\eqref{eqn:supint.generators}. Note that the essential metrizability space, $\essSp\subset\metrSp$, is the complement (in $\metrSp$) of the union of the three eigenspaces corresponding to the~$\g{i}$.
Since we are interested in normal forms up to isometries, we can identify diffeomorphic metrics within the 3-dimensional space $\solSp$ (respectively, within~$\metrSp$).
Thus, we end up with the normal form
\begin{equation}\label{eqn:dom.3.spherical.normal.form}
\sigma[\theta,\varphi]
=\sin(\theta)\,\cos(\varphi)\,\sigma_1
+\sin(\theta)\,\sin(\varphi)\,\sigma_2
+\cos(\theta)\,\sigma_3\,,
\end{equation}
with $\theta\in(0,\pi)$ and $\varphi\in[0,2\pi)$ where we require $\varphi\notin\{0,\frac{\pi}{2},\pi,\frac{3\pi}{2}\}$ if $\theta=\frac{\pi}{2}$.
Note that~\eqref{eqn:dom.3.spherical.normal.form} amounts to a parameterization in terms of points on the 2-sphere.

Let us briefly contrast the normal forms~\eqref{eqn:dom.3.explicit.normal.form} and~\eqref{eqn:dom.3.spherical.normal.form} from a geometric viewpoint.
Obviously, the metrics~\eqref{eqn:supint.generators} provide a basis of the 3-dimensional space~$\solSp$. We should then remove the eigenspaces of $\lie_X$, to arrive at $\essSp=\solSp\setminus\{\text{eigenspaces of $\lie_X$}\}$.
At this step, many metrics will still be diffeomorphic, so we need to remove this remaining ambiguity to arrive at proper normal forms. Reference~\cite{manno_2018} provides two approaches for this problem, resulting in either~\eqref{eqn:dom.3.explicit.normal.form} or~\eqref{eqn:dom.3.spherical.normal.form}. It is not explained in the reference, so let us briefly comment on the underlying geometric procedure.

As is proven in~\cite{manno_2018}, the orbits in~$\solSp$ under isometries are precisely given by the action of the projective group, generated by~$X$. 
We fix normal forms by specifying 2-dimensional surfaces in $\solSp$ that have (at most) one intersection with each orbit. Let us start with the normal forms~\eqref{eqn:dom.3.spherical.normal.form}, because this is conceptually more straightforward. Indeed, consider the 2-sphere $\mathbb{S}^2\subset\mathbb{R}^3$. It has precisely one intersection with each orbit, and these intersections are precisely the points~\eqref{eqn:dom.3.spherical.normal.form}.


On the other hand, we might choose planes in $\solSp\sim\mathbb{R}^3$. Let $\sigma=\sum_iu_{i}\sigma_i\in\solSp$. Then consider the plane $\{u_3=0\}$. On this plane, which actually is a 2-dimensional $\lie_X$-invariant subspace of~$\solSp$, the action of~$X$ simplifies to
\[
 \phi^*_t(u_1\sigma_1+u_2\sigma_2)
 =e^{\lambda_1 t}u_1\sigma_1+e^{\lambda_2 t}u_2\sigma_2\,,
\]
where $\lambda_1$ and $\lambda_2$ are the eigenvalues of $\lie_X$ for $\sigma_1$ and $\sigma_2$, respectively. Since we are not interested in cases where either $u_1=0$ or $u_2=0$ (because then $X$ would become homothetic), and we cannot change the signs of each coordinate, we are able to choose the line $u_2\mp u_1=0$ to obtain the normal forms. Thus,
\[
 \sigma = \kappa\,(\sigma_1\pm\sigma_2)
\]
where $\kappa=u_1=\pm u_2$. This is (basically) the normal form~\eqref{eqn:dom.3.nf.1}.
The normal forms~\eqref{eqn:dom.3.nf.2} and~\eqref{eqn:dom.3.nf.3} can be obtained by similar considerations.

\section*{Acknowledgements}
The authors would like to thank Vladimir Matveev and Stefan Rosemann, as well as Rod Gover and Jonathan Kress, for fruitful discussions and valuable comments.
Gianni Manno (GM) acknowledges  that the present research has been partially supported by the following projects: ``Connessioni proiettive, equazioni di Monge-Amp\`ere e sistemi integrabili'' by Istituto Nazionale di Alta Matematica (INdAM),  ``MIUR grant Dipartimenti di Eccellenza 2018-2022 (E11G18000350001)'' and ``Finanziamento alla ricerca 2017-2018 (53\_RBA17MANGIO)''. GM is a member of INdAM.
Andreas Vollmer (AV) is a postdoc research fellow of Deutsche Forschungsgemeinschaft (DFG) through the project ``Second and third order superintegrable systems: classification and applications'' funded by Deutsche Forschungsgemeinschaft (DFG, German Research Foundation); project number 353063958.
AV would like to thank Friedrich Schiller University (FSU) Jena, the University of Pavia and the University of Auckland for hospitality and acknowledges travel support from DFG, FSU, INdAM and the University of Auckland.

\sloppy
\printbibliography

\end{document}